\documentclass{amsart}
\usepackage{amssymb}
\usepackage{braket}
\usepackage{mathrsfs}
\usepackage[all]{xy}
\usepackage{graphicx}
\usepackage{stmaryrd}

\newtheorem{thm}{Theorem}[section]
\newtheorem{cor}[thm]{Corollary}
\newtheorem{prop}[thm]{Proposition}
\theoremstyle{definition}
\newtheorem{dfn}[thm]{Definition}
\newtheorem{ex}[thm]{Example}
\newtheorem{claim}[thm]{Claim}
\newtheorem{lem}[thm]{Lemma}

\theoremstyle{remark}
\newtheorem{rem}[thm]{Remark}
\newtheorem{introdfn}{Definition}

\newcommand{\Ob}{\mathrm{Ob}}         
\newcommand{\id}{\mathrm{id}}         
\newcommand{\Id}{\mathrm{Id}}         
\newcommand{\ppr}{^{\prime}}          
\newcommand{\pprr}{^{\prime\prime}}   
\newcommand{\sh}{\sharp}              
\newcommand{\fa}{\forall}             
\newcommand{\am}{\amalg}              
\newcommand{\co}{\colon}              
\newcommand{\ci}{\circ}               
\newcommand{\iv}{^{-1}}               
\newcommand{\uas}{^{\ast}}            
\newcommand{\sas}{_{\ast}}            
\newcommand{\bs}{\backslash}          

\newcommand{\Mon}{\mathit{Mon}}     
\newcommand{\Sett}{\mathit{Set}}    
\newcommand{\Ab}{\mathit{Ab}}       
\newcommand{\Add}{\mathit{Add}}     
\newcommand{\Sadd}{\mathit{Sadd}}   
\newcommand{\Mod}{\mathit{Mod}}     
\newcommand{\RMod}{R\Mod}           
\newcommand{\Fun}{\mathit{Fun}}     

\newcommand{\lla}{\longleftarrow}     
\newcommand{\lra}{\longrightarrow}    
\newcommand{\tc}{\Rightarrow}         
\newcommand{\LR}{\Leftrightarrow}     
\newcommand{\thra}{\twoheadrightarrow}

\newcommand{\al}{\alpha}         
\newcommand{\be}{\beta}          
\newcommand{\lam}{\lambda}       
\newcommand{\kp}{\kappa}         
\newcommand{\ups}{\upsilon}      
\newcommand{\sig}{\sigma}        
\newcommand{\ep}{\varepsilon}    
\newcommand{\thh}{\theta}        
\newcommand{\vp}{\varphi}        
\newcommand{\Th}{\Theta}         

\newcommand{\Lam}{\Lambda}       

\newcommand{\Csc}{\mathscr{C}}  
\newcommand{\Cbb}{\mathbb{C}}   
\newcommand{\Fbb}{\mathbb{F}}   
\newcommand{\Rbb}{\mathbb{R}}   
\newcommand{\Sbb}{\mathbb{S}}   
\newcommand{\Zbb}{\mathbb{Z}}   
\newcommand{\rbf}{\mathbf{r}}   
\newcommand{\tbf}{\mathbf{t}}   
\newcommand{\Bcal}{\mathcal{B}} 
\newcommand{\Gcal}{\mathcal{G}} 
\newcommand{\Afr}{\mathfrak{A}} 
\newcommand{\afr}{\mathfrak{a}} 
\newcommand{\bfr}{\mathfrak{b}} 
\newcommand{\cfr}{\mathfrak{c}} 

\newcommand{\wt}{\widetilde}    
\newcommand{\und}{\underline}   
\newcommand{\ovl}{\overline}    
\newcommand{\ov}{\overset}      
\newcommand{\un}{\underset}     

\newcommand{\SMack}{\mathit{SMack}}  
\newcommand{\Mack}{\mathit{Mack}}    

\newcommand{\HUG}{{}_HU_G}           
\newcommand{\HUK}{{}_HU_K}           
\newcommand{\LVH}{{}_LV_H}           
\newcommand{\VU}{V\un{H}{\times}U}   
\newcommand{\HHG}{{}_HH_G}           
\newcommand{\GHH}{{}_GH_H}           
\newcommand{\Inddef}{\mathrm{Inddef}}
\newcommand{\Infres}{\mathrm{Infres}}

\newcommand{\pt}{\mathbf{1}}                     
\newcommand{\dfl}{\mathrm{dfl}}                  
\newcommand{\sika}{\boxplus}                     
\newcommand{\maru}{\oplus}                       
\newcommand{\Obig}{\Omega_{\mathrm{big}}}        
\newcommand{\msp}{\mathrm{sp}}                   

\newcommand{\FG}{\mathrm{FinGrp}}             
\newcommand{\sFG}{\und{\FG}}                  
\newcommand{\SMackS}{\mathit{SMack}(\Sbb)}    
\newcommand{\MackS}{\mathit{Mack}(\Sbb)}      
\newcommand{\MackSR}{\mathit{Mack}^R(\Sbb)}   
\newcommand{\MackdSR}{\Mack_{\dfl}^R(\Sbb)}   
\newcommand{\AddC}{\Add(\Csc)}                
\newcommand{\AddCR}{\Add^R(\Csc)}             
\newcommand{\SaddC}{\Sadd(\Csc)}              
\newcommand{\ResC}{\mathit{Res}(\Csc)}        
\newcommand{\ResCR}{\mathit{Res}^R(\Csc)}     
\newcommand{\BisetFtr}{\mathit{BisetFtr}}    
\newcommand{\Sxg}{\Sbb/\!_{\xg}}             
\newcommand{\ESxg}{E\text{-}\Sxg}            
\newcommand{\FFA}{\Fun(\sFG,\RMod)}            

\newcommand{\xg}{\frac{X}{G}}
\newcommand{\xgp}{\frac{X\ppr}{G\ppr}}

\newcommand{\yg}{\frac{Y}{G}}
\newcommand{\yh}{\frac{Y}{H}}
\newcommand{\yhp}{\frac{Y\ppr}{H\ppr}}
\newcommand{\zk}{\frac{Z}{K}}
\newcommand{\ak}{\frac{A}{K}}
\newcommand{\akp}{\frac{A\ppr}{K\ppr}}

\newcommand{\bl}{\frac{B}{L}}
\newcommand{\blp}{\frac{B\ppr}{L\ppr}}

\newcommand{\ch}{\frac{C}{H}}
\newcommand{\wl}{\frac{W}{L}}

\newcommand{\ptg}{\frac{\pt}{G}}
\newcommand{\pth}{\frac{\pt}{H}}
\newcommand{\ptf}{\frac{\pt}{f}}

\newcommand{\althh}{\frac{\alpha}{\theta}}
\newcommand{\althhp}{\frac{\alpha\ppr}{\theta\ppr}}
\newcommand{\bet}{\frac{\beta}{\tau}}
\newcommand{\betp}{\frac{\beta\ppr}{\tau\ppr}}

\newcommand{\akaxg}{\ak\ov{\afr}{\to}\xg}
\newcommand{\akaxgp}{\akp\ov{\afr\ppr}{\to}\xg}
\newcommand{\blbxg}{\bl\ov{\bfr}{\to}\xg}
\newcommand{\blbyh}{\bl\ov{\bfr}{\to}\yh}

\newcommand{\chcxg}{\ch\ov{\cfr}{\to}\xg}

\newcommand{\xgixg}{\xg\ov{\id}{\to}\xg}

\newcommand{\yhiyh}{\yh\ov{\id}{\to}\yh}
\newcommand{\akiak}{\ak\ov{\id}{\to}\ak}

\newcommand{\kfgk}{K\ov{f}{\to}G,\kp}
\newcommand{\fk}{f,\kp}

\newcommand{\Pm}{P_{\maru}}

\numberwithin{equation}{section}

\begin{document}

\title[Several adjoint constructions for biset functors]{Several adjoint constructions for biset functors via Mackey-functorial interpretation.}

\author{Hiroyuki NAKAOKA}
\address{Department of Mathematics and Computer Science, Kagoshima University, 1-21-35 Korimoto, Kagoshima, 890-0065 Japan\ /\ LAMFA, Universit\'{e} de Picardie-Jules Verne, 33 rue St Leu, 80039 Amiens Cedex1, France}

\email{nakaoka@sci.kagoshima-u.ac.jp}
\urladdr{http://www.lamfa.u-picardie.fr/nakaoka/}

\thanks{The author wishes to thank Professor Laurence John Barker for stimulating arguments}
\thanks{The author wishes to thank Professor Serge Bouc for his interest and comments}

\thanks{This work is supported by JSPS Grant-in-Aid for Young Scientists (B) 25800022, JSPS Grant-in-Aid for Scientific Research (C) 24540085}

\begin{abstract}
We consider analogs of Jacobson's $F$-Burnside construction and Boltje's $(-)_+$-construction for biset functors, using Mackey-functor theoretic interpretation of biset functors.
\end{abstract}

\maketitle

\tableofcontents


\section{Introduction and Preliminaries}

In the previous article \cite{N_BisetMackey}, we constructed a 2-category $\Sbb$ of finite sets with variable finite group actions. Using 2-coproducts and 2-fibered products in $\Sbb$, we can define the notion of a Mackey functor on the classifying category $\Csc$ of $\Sbb$. In \cite{N_BisetMackey}, it has been shown that biset functors can be regarded as special class of these Mackey functors, and characterized among them by the condition which we call '{\it deflativity}'.

It can be expected that, several functorial constructions for ordinary Mackey functors will be also performed analogously on this category $\Csc$, and provide some constructions which will serve to developments in biset functor theory.

This article is devoted to a demonstration of this machinery, through producing analogs of Jacobson's $F$-Burnside construction and Boltje's $(-)_+$-construction for biset functors.
The following diagram indicates the relations of the categories and functors introduced in this article. Each curved arrow is left adjoint to the functor in the opposite direction. 
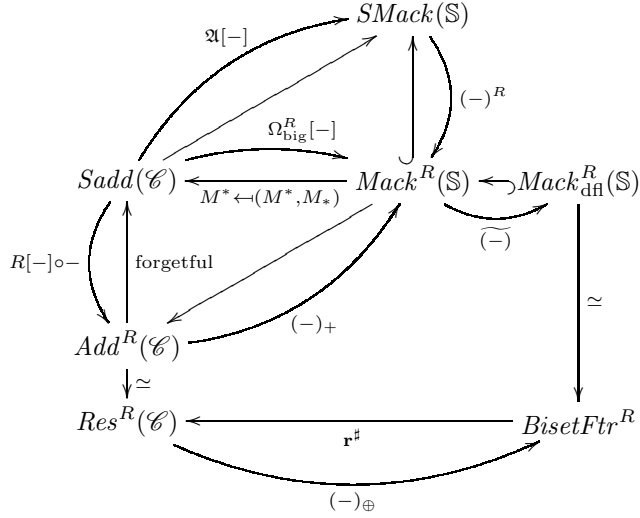
\[
\xy
(-30,20)*+{\SaddC}="0";
(-30,-2)*+{\AddCR}="2";
(-30,-12)*+{\ResCR}="4";
(8,42)*+{\SMackS}="6";
(8,20)*+{\MackSR}="8";
(30,20)*+{\MackdSR}="10";
(30,-12)*+{\BisetFtr^R}="12";
{\ar_{\text{forgetful}} "2";"0"};
{\ar^{M\uas\mapsfrom (M\uas,M\sas)} "8";"0"};
{\ar^{} "0";"6"};
{\ar^{\simeq} "2";"4"};
{\ar^{\rbf^{\sh}} "12";"4"};
{\ar@{^(->} "8";"6"};
{\ar@{^(->} "10";"8"};
{\ar^{\simeq} "10";"12"};
{\ar_{} "8";"2"};
{\ar@/_1.20pc/_{R [-]\ci-} "0";"2"};
{\ar@/^1.0pc/^(0.6){\Obig^R [-]} "0";"8"};
{\ar@/^1.60pc/^{\Afr [-]} "0";"6"};
{\ar@/^1.20pc/^{(-)^R} "6";"8"};
{\ar@/_1.60pc/_{(-)_+} "2";"8"};
{\ar@/_1.20pc/_{\wt{(-)}} "8";"10"};
{\ar@/_2.0pc/_{(-)_{\maru}} "4";"12"};
\endxy
\]
In this diagram, $\MackSR$ (resp. $\SMackS$) denotes the category of $R$-linear (resp. semi-)Mackey functors on $\Sbb$, where $R$ is a commutative coefficient ring. The full subcategory of deflative Mackey functors is denoted by $\MackdSR$. The category of $R$-linear biset functors is denoted by $\BisetFtr^R$. In \cite{N_BisetMackey}, the equivalence $\MackdSR\ov{\simeq}{\lra}\BisetFtr^R$ has been constructed. The definitions of these categories are reviewed in section 2.

The other categories and functors in the above diagram are defined in the succeeding sections. We summarize the results in each section here. Throughout, we use the following notation.
\begin{itemize}
\item[-] $\Sett$ denotes the category of sets and maps.
\item[-] $\FG$ denotes the category of finite groups and group homomorphisms.
\item[-] $\RMod$ denotes the category of $R$-modules and $R$-homomorphisms, for a fixed coefficient ring $R$.
\end{itemize}

In section 3, we 
construct a functor
\[ \wt{(-)}\co\MackSR\to \MackdSR, \]
which is the left adjoint of the inclusion $\MackdSR\hookrightarrow\MackSR$. 

In section 4, we consider an analog of Jacobson's $F$-Burnside construction (\cite{Jacobson}) for $\MackSR$. From the category $\SaddC$ of contravariant functors $E\co \Csc\to\Sett$ sending finite products to coproducts, we construct functors
\[ \Afr [-]\co\SaddC\to \SMackS\quad\text{and}\quad \Obig^R[-]\co\SaddC\to\MackSR, \]
which are left adjoint to the forgetful functors 
$\SMackS\to\SaddC$ and $\MackSR\to\SaddC$.

In section 5, we consider an analog of Boltje's $(-)_+$-construction (\cite{Boltje}) for $\MackSR$. From the category $\AddCR$ of contravariant functors $F\co \Csc\to\RMod$ sending finite products to coproducts, we construct a functor
\[ (-)_+\co\AddCR\to \MackSR, \]
which is left adjoint to the forgetful functor 
$\MackSR\to\AddCR$.

In section 6, we show the equivalence $\AddCR\ov{\simeq}{\lra}\ResCR$. 
Here, $\ResCR$ denotes the category of functors $P\co\FG\to\RMod$ satisfying $P(\sig_g)=\id_{P(G)}$ for any finite group $G$ and $g\in G$, where 
$\sigma_g\colon G\to G$ 
is the conjugation map.
Composing with the functors constructed in section 3 and 5, we obtain a functor
\[ (-)_{\maru}\co\ResCR\to\BisetFtr^R, \]
which is left adjoint to a natural functor $\rbf^{\sh}\co\BisetFtr^R\to\ResCR$.

In section 7, for the convenience of the users of biset functor theory, we introduce the direct construction of the functor $(-)_{\maru}$. This construction involves essentially left Kan extension.

\medskip

Throughout this article, any group is assumed to be finite. 
The unit of a group will be denoted by $e$. 
A one-point set is denoted by $\pt$, on which any finite group $G$ acts 
trivially. 
A biset is always assumed to be finite.
A monoid is always assumed to be unitary and commutative. Similarly a ring is assumed to be commutative, with an additive unit $0$ and a multiplicative unit $1$. A monoid homomorphism preserves units.
We denote the category of monoids by $\Mon$.

For any category $\mathscr{K}$ and any pair of objects $X$ and $Y$ in $\mathscr{K}$, the set of morphisms from $X$ to $Y$ in $\mathscr{K}$ is denoted by $\mathscr{K}(X,Y)$. 
Any 2-category is assumed to be strict (\cite{Borceux},\cite{MacLane}). For a 2-category $\Cbb$, the entity of 0-cells (respectively 1-cells, 2-cells) is denoted by $\Cbb^0$ (resp. $\Cbb^1$, $\Cbb^2$). For a pair of 0-cells $X, Y$ in $\Cbb$, the set of 1-cells from $X$ to $Y$ is denoted by $\Cbb^1(X,Y)$.

\section{Review of the definitions}

We review the definitions and results from \cite{N_BisetMackey}. Details can be found in \cite{N_BisetMackey}.

\smallskip

The 2-category of finite sets with variable group actions is defined as follows.
\begin{dfn}\label{DefS}
2-category $\Sbb$ is defined as follows.
\begin{enumerate}
\item[{\rm (0)}] A 0-cell is a pair of a finite group $G$ and a finite $G$-set $X$. We denote this pair by $\xg$.
\item[{\rm (1)}] For any pair of 0-cells $\xg$ and $\yh$, a morphism $\althh\co \xg\to\yh$ is a pair of a map $\al\co X\to Y$ and a family of maps $\{\thh_x\co G\to H \}_{x\in X}$ satisfying
\begin{itemize}
\item[{\rm (i)}] $\al(gx)=\thh_x(g)\al(x)$ 
\item[{\rm (ii)}] $\thh_x(gg\ppr)=\thh_{g\ppr x}(g)\thh_x(g\ppr)$
\end{itemize}
for any $x\in X$ and any $g,g\ppr\in G$.
\item[{\rm (2)}] For any pair of 1-cells $\althh,\althhp\co\xg\to\yh$, a 2-cell $\ep\co\althh\tc\althhp$ is a family of elements $\{ \ep_x\in H\}_{x\in X}$ satisfying
\begin{itemize}
\item[{\rm (i)}] $\al\ppr(x)=\ep_x\al(x) $,
\item[{\rm (ii)}] $\ep_{gx}\thh_x(g)\ep_x\iv=\thh\ppr_x(g)$
\end{itemize}
for any $x\in X$ and $g\in G$.
\end{enumerate}
A 1-cell $\althh$ is often abbreviately written as $\al$. Remark that a 1-cell $\al\co \xg\to\yh$ preserves orbits. Namely, for any $x\in X$ we have $\al(Gx)\subseteq H\al(x)$.
\end{dfn}
Horizontal composition is denoted by \lq\lq$\ci$", while \lq\lq$\cdot$" denotes vertical composition. For example, for any diagram
\[
\xy
(-28,0)*+{\xg}="0";
(0,0)*+{\yh}="2";
(28,0)*+{\zk}="4";
{\ar@/^1.2pc/^{\althh} "0";"2"};
{\ar@/_1.2pc/_{\althhp} "0";"2"};
{\ar@/^1.2pc/^{\bet} "2";"4"};
{\ar@/_1.2pc/_{\betp} "2";"4"};
{\ar@{=>}^{\ep} (-14,2);(-14,-2)};
{\ar@{=>}^{\delta} (14,2);(14,-2)};
\endxy
\]
in $\Sbb$, we have an equality
\[ (\delta\ci\al\ppr)\cdot(\be\ci\ep)=(\be\ppr\ci\ep)\cdot(\delta\ci\al). \]

\begin{rem}
$\ $
\begin{enumerate}
\item For a fixed finite group $G$, a $G$-map $\al\co X\to Y$ induces a 1-cell $\althh\co\xg\to\yg$, with $\thh=\{\id_G\co G\to G\}_{x\in X}$. we denote this 1-cell by $\frac{\al}{G}$.
\item Any homomorphism of finite groups $f\co G\to H$ induces a 1-cell $\frac{\pt}{f}\co\ptg\to\pth$.
\end{enumerate}
\end{rem}

Category $\Csc$ is defined to be the classifying category of $\Sbb$, as follows.
\begin{dfn}\label{DefC}
Category $\Csc$ is defined as follows.
\begin{itemize}
\item[{\rm (i)}] $\Ob(\Csc)=\Sbb^0$.
\item[{\rm (ii)}] For any pair of objects $\xg,\yh\in\Ob(\Csc)$, we define an equivalence relation on $\Sbb^1(\xg,\yh)$ as follows.
\begin{itemize}
\item[-] 1-cells $\althh,\althhp\in\Sbb^1(\xg,\yh)$ are equivalent if there exists some 2-cell $\ep\co \althh\tc\althhp$.
\end{itemize}
The set of morphisms $\Csc(\xg,\yh)$ is defined to be the quotient of $\Sbb^1(\xg,\yh)$ by this equivalence:
\[ \Csc(\xg,\yh)=\Sbb^1(\xg,\yh)\Big/\text{2-cells} \]
The equivalence class of $\althh$ is denoted by $\und{\big(\althh\big)}$, or simply by $\und{\al}$.
\end{itemize}
\end{dfn}

A 1-cell $\al\co\xg\to\yh$ is said to be an {\it adjoint equivalence} if there is a 1-cell $\be\co\yh\to\xg$ and 2-cells $\rho\co \be\ci\al\tc\id$, $\lam\co \al\ci\be\tc\id$, which satisfy
\[ \al\ci\rho=\lam\ci\al,\ \ \rho\ci\be=\be\ci\lam. \]
This $\be$ is called a {\it quasi-inverse} of $\al$, and denoted by $\be=\al\iv$.
Remark that if $\al$ is an adjoint equivalence in $\Sbb$, then $\und{\al}$ becomes an isomorphism in $\Csc$. 

\begin{ex}
$\ \ $
\begin{enumerate}
\item If $\al\co X\to Y$ is an isomorphism of finite $G$-sets, then $\frac{\al}{G}\co\xg\ov{\simeq}{\lra}\yg$ is an adjoint equivalence in $\Sbb$.
\item If $f\co G\to H$ is an isomorphism of finite groups, then $\frac{\pt}{f}\co\ptg\ov{\simeq}{\lra}\pth$ is an adjoint equivalence in $\Sbb$.
\end{enumerate}
\end{ex}

\bigskip

The following results have been shown in \cite{N_BisetMackey}.

\begin{dfn}
For any $\xg$ and $\yh$ in $\Sbb^0$, their {\it 2-coproduct} $(\xg\am \yh,\ups_{\xg},\ups_{\yh})$ is defined to be a triplet of $\xg\am \yh\in\Sbb^0$ and $\ups_{\xg}\in\Sbb^1(\xg, \xg\am \yh)$, $\ups_{\yh}\in\Sbb^1(\yh,\xg\am \yh)$, satisfying the following conditions.
\begin{itemize}
\item[{\rm (i)}]
For any $\wl\in\Sbb^0$ and $\al\in\Sbb^1(\xg,\wl),\,\be\in\Sbb^1(\yh,\wl)$, there exist $\al\cup\be\in\Sbb^1(\xg\am \yh,\wl)$ and 2-cells $\xi,\eta$ as in the following diagram.
\[
\xy
(0,-6)*+{\wl}="0";
(-18,10)*+{\xg}="2";
(0,10)*+{\xg\am \yh}="4";
(18,10)*+{\yh}="6";
{\ar_{\al} "2";"0"};
{\ar|*+{_{\al\cup\be}} "4";"0"};
{\ar^{\be} "6";"0"};
{\ar^(0.4){\ups_{\xg}} "2";"4"};
{\ar_(0.4){\ups_{\yh}} "6";"4"};
{\ar@{=>}_{\xi} (-4,6);(-7.5,3)};
{\ar@{=>}^{\eta} (4,6);(7.5,3)};
\endxy
\]

\item[{\rm (ii)}]
For any triplets $(\gamma,\xi\ppr,\eta\ppr)$ as in {\rm (i)}, there exists a unique 2-cell $\zeta\co \al\cup\be\tc\gamma$ which satisfies $\xi\ppr\cdot(\zeta\ci\ups_{\xg})=\xi$ and $\eta\cdot\ppr(\zeta\ci\ups_{\yh})=\eta$.
\end{itemize}
By its universality, the 2-coproduct is determined up to adjoint equivalences, if it exists.
\end{dfn}

\begin{rem}
If $(\xg\am \yh,\ups_{\xg},\ups_{\yh})$ is a 2-coproduct of $\xg$ and $\yh$ in $\Sbb$, then $(\xg\am \yh,\und{\ups_{\xg}},\und{\ups_{\yh}})$ gives a coproduct of $\xg$ and $\yh$ in $\Cbb$.
\end{rem}

\begin{prop}\label{Prop2CoprodVari}
For any pair of 0-cells $\xg$ and $\yh$ in $\Sbb$, their 2-coproduct exists. 
Moreover if $G=H$, then $\xg\am\yg$ is given by $\frac{X\am Y}{G}$, where $X\am Y$ is the usual disjoint union of $G$-sets.
\end{prop}

\begin{prop}\label{PropIndEquiv}
Let $\xg$ be any 0-cell in $\Sbb$. For any $x\in X$, if we denote the stabilizer by $G_x$ and the orbit by $Gx$, then there is a natural adjoint equivalence
\[ \zeta_x\co \frac{\pt}{G_x}\ov{\simeq}{\lra}\frac{Gx}{G}. \]
If we take a set of representatives $x_1,\ldots,x_s\in X$ of $G$-orbits, then the 1-cell obtained by the universality of the 2-coproduct
\[ \zeta=\underset{1\le i\le s}{\bigcup}\zeta_{x_i}\co \coprod_{1\le i\le s}\frac{\pt}{G_{x_i}}\ov{\simeq}{\lra}\coprod_{1\le i\le s}\frac{Gx_i}{G}\ov{\simeq}{\lra}\xg \]
gives an adjoint equivalence.
\end{prop}

\begin{dfn}\label{Def2Pullback}
For any pair of 1-cells
\[ \xg\ov{\al}{\lra}\zk\ov{\be}{\lla}\yh \]
in $\Sbb$, its {\it 2-fibered product} of $\al$ and $\be$ is defined to be a quartet $(\xg\times_{\zk}\yh,\gamma,\delta,\kappa)$ as in the diagram
\begin{equation}\label{Diag_2FibProd}
\xy
(-9,6)*+{\xg\times_{\zk}\yh}="0";
(9,6)*+{\yh}="2";
(-9,-6)*+{\xg}="4";
(9,-6)*+{\zk}="6";
{\ar^(0.6){\delta} "0";"2"};
{\ar_{\gamma} "0";"4"};
{\ar^{\be} "2";"6"};
{\ar_{\al} "4";"6"};
{\ar@{=>}_{\kappa} (-2,0);(2,0)};
\endxy
,
\end{equation}
which satisfies the following conditions.
\begin{itemize}
\item[{\rm (i)}]
For any diagram in $\Sbb$
\[
\xy
(-8,6)*+{\wl}="0";
(8,6)*+{\yh}="2";
(-8,-6)*+{\xg}="4";
(8,-6)*+{\zk}="6";
{\ar^{\psi} "0";"2"};
{\ar_{\vp} "0";"4"};
{\ar^{\be} "2";"6"};
{\ar_{\al} "4";"6"};
{\ar@{=>}^{\ep} (-2,0);(2,0)};
\endxy
,
\]
there exist $\pi,\xi,\eta$ as in the diagram
\[
\xy
(-22,16)*+{\wl}="-2";
(-8,6)*+{\xg\times_{\zk}\yh}="0";
(8,6)*+{\yh}="2";
(-8,-7)*+{\xg}="4";
(8,-7)*+{\zk}="6";
{\ar^{\pi} "-2";"0"};
{\ar@/^1.24pc/^{\psi} "-2";"2"};
{\ar@/_1.24pc/_(0.68){\vp} "-2";"4"};
{\ar_(0.66){\delta} "0";"2"};
{\ar^{\gamma} "0";"4"};
{\ar^{\be} "2";"6"};
{\ar_{\al} "4";"6"};
{\ar@{=>}_{\kappa} (-1.5,-1);(2.5,-1)};
{\ar@{=>}^{\xi} (-12,3);(-16,-1)};
{\ar@{=>}_{\eta} (-8,9);(-6,14)};
\endxy
,
\]
satisfying $\ep\cdot(\al\ci\xi)=(\be\ci\eta)\cdot(\kappa\ci \pi)$.
\item[{\rm (ii)}]
For any triplets $(\pi\ppr,\vp\ppr,\psi\ppr)$ as in {\rm (i)}, there exists a unique 2-cell $\zeta\co\pi\tc\pi\ppr$ which satisfies $\xi\ppr\cdot(\gamma\ci\zeta)=\xi$ and $\eta\ppr\cdot(\delta\ci\zeta)=\eta$
\end{itemize}
By its universality, the 2-fibered product is determined up to adjoint equivalences, if it exists.
\end{dfn}

\begin{prop}\label{Prop2Pullback}
For any pair of 1-cells
\[ \xg\ov{\al}{\lra}\zk\ov{\be}{\lla}\yh, \]
its 2-fibered product exists in $\Sbb$.
\end{prop}

\begin{rem}
Even if $(\ref{Diag_2FibProd})$ is a 2-fibered product in $\Sbb$, its image in $\Csc$
\[
\xy
(-9,6)*+{\xg\times_{\zk}\yh}="0";
(9,6)*+{\yh}="2";
(-9,-6)*+{\xg}="4";
(9,-6)*+{\zk}="6";
{\ar^(0.6){\und{\delta}} "0";"2"};
{\ar_{\und{\gamma}} "0";"4"};
{\ar^{\und{\be}} "2";"6"};
{\ar_{\und{\al}} "4";"6"};
{\ar@{}|\circlearrowright "0";"6"};
\endxy
\]
is not necessarily a fibered product in $\Csc$. In fact, this is only a weak fibered product.
Nevertheless, these weak fibered products which come from 2-fibered products are closed under isomorphisms in $\Csc$, and thus form a natural distinguished class among weak fibered products.
\end{rem}

\begin{dfn}\label{DefStabsurj}
A 1-cell $\al\co\xg\to\yh$ is called {\it surjective on stabilizers} or shortly {\it stab-surjective}, if the following conditions are satisfied.
\begin{itemize}
\item[{\rm (i)}] $Y=H\al(X)$ holds.
\item[{\rm (ii)}] If $x,x\ppr\in X$ and $h,h\ppr\in H$ satisfy $h\al(x)=h\ppr\al(x\ppr)$, then there exists $g\in G$ which satisfies $x\ppr=gx$ and $h=h\ppr\thh(g)$.
\end{itemize}
\end{dfn}

\begin{prop}\label{PropStabsurj}
The following holds for the stab-surjectivity.
\begin{enumerate}
\item If $\al$ is an adjoint equivalence, then $\al$ is stab-surjective.
\item Stab-surjectivity is closed under equivalences of 1-cells. Namely, if there exists a 2-cell $\ep\co\al\tc\al\ppr$ and if $\al$ is stab-surjective, then so is $\al\ppr$.
\item Stab-surjectivity is closed under compositions of 1-cells. Namely, if $\xg\ov{\al}{\lra}\yh\ov{\be}{\lra}\zk$ is a sequence of 1-cells and if $\al$ and $\be$ are stab-surjective, then so is $\be\ci\al$.
\item Stab-surjectivity is closed under 2-pullbacks. Namely, if
\[
\xy
(-8,6)*+{\wl}="0";
(8,6)*+{\yh}="2";
(-8,-6)*+{\xg}="4";
(8,-6)*+{\zk}="6";
{\ar^{\delta} "0";"2"};
{\ar_{\gamma} "0";"4"};
{\ar^{\be} "2";"6"};
{\ar_{\al} "4";"6"};
{\ar@{=>}^{\ep} (-2,0);(2,0)};
\endxy
\]
is a 2-fibered product in $\Sbb$ and if $\be$ is stab-surjective, then so is $\gamma$.
\end{enumerate}
\end{prop}

\begin{dfn}\label{Def2Slice}
Let $\xg$ be any 0-cell in $\Sbb$. Then a 2-category $\Sxg$ is defined as follows.
\begin{enumerate}
\item[{\rm (0)}] A 0-cell in $\Sxg$ is a 1-cell $(\akaxg)$ in $\Sbb$, from some $\ak\in\Sbb^0$.
\item[{\rm (1)}] A 1-cell in $\Sxg$ from $(\akaxg)$ to $(\blbxg)$ is a pair $(\vp,\mu)$ of a 1-cell $\vp$ and a 2-cell $\mu$ in $\Sbb$ as in the following diagram.
\[
\xy
(-10,6)*+{\ak}="0";
(10,6)*+{\bl}="2";
(0,-8)*+{\xg}="4";
{\ar^{\vp} "0";"2"};
{\ar_(0.4){\al} "0";"4"};
{\ar^(0.38){\be} "2";"4"};
{\ar@{=>}^{\mu} (2,3);(-2,0)};
\endxy
\]
\item[{\rm (2)}] If $(\vp,\mu)\co (\akaxg)\to(\blbxg)$ and $(\vp\ppr,\mu\ppr)\co(\akaxg)\to(\blbxg)$ are 1-cells in $\Sxg$, then a 2-cell $\ep\co (\vp,\mu)\tc(\vp\ppr,\mu\ppr)$ in $\Sxg$ is a 2-cell $\ep\co \vp\tc \vp\ppr$ in $\Sbb$, which makes the following diagram commutative.
\[
\xy
(-10,6)*+{\be\ci \vp}="0";
(10,6)*+{\be\ci \vp}="2";
(0,-8)*+{\al}="4";
(0,10)*+{}="5";
{\ar@{=>}^{\be\ci\ep} "0";"2"};
{\ar@{=>}_(0.4){\mu} "0";"4"};
{\ar@{=>}^(0.4){\mu\ppr} "2";"4"};
{\ar@{}|\circlearrowright "4";"5"};
\endxy
\]
\end{enumerate}

Composition of 1-cells
\[ (\akaxg)\ov{(\varphi,\mu)}{\lra}(\blbxg)\ov{(\psi,\nu)}{\lra}(\chcxg) \]
is defined to be
\[ (\psi\ci \varphi,\mu\cdot(\nu\ci \varphi))\co (\akaxg)\to(\chcxg). \]
Vertical composition of 2-cells
\[
\xy
(-20,0)*+{(\akaxg)}="0";
(20,0)*+{(\blbxg)}="2";
{\ar@/^2.0pc/^{(\varphi,\mu)} "0";"2"};
{\ar|*+{_{(\varphi\ppr,\mu\ppr)}} "0";"2"};
{\ar@/_2.0pc/_{(\varphi\pprr,\mu\pprr)} "0";"2"};
{\ar@{=>}^{\ep} (0,6);(0,3)};
{\ar@{=>}^{\ep\ppr} (0,-3);(0,-6)};
\endxy
\]
is defined to be $\ep\ppr\cdot\ep$, using the vertical composition in $\Sbb$.
Horizontal composition is also defined by using the horizontal composition in $\Sbb$.
\end{dfn}

\begin{rem}
For any $\xg\in\Sbb^0$, the 2-category $\Sxg$ has 2-coproducts and 2-products, induced from 2-coproducts and 2-fibered products in $\Sbb$.
If we denote the set of adjoint equivalence classes of 0-cells in $\Sxg$ by $\Afr(\xg)$, then it has a structure of commutative semi-ring, with the addition induced from 2-coproducts and the multiplication induced from 2-fibered products. 
\end{rem}

\begin{dfn}
For any $\xg\in\Sbb^0$, the additive ring completion of $\Afr(\xg)$ is denoted by $\Obig(\xg)$, and called the {\it bigger Burnside ring}.
\end{dfn}

\begin{dfn}\label{DefSemiMackS}
A {\it semi-Mackey functor} $M=(M^{\ast},M_{\ast})$ on $\Sbb$ is a pair of a contravariant 2-functor
\[ M^{\ast}\co\Sbb\to\Sett \]
and a covariant 2-functor
\[ M_{\ast}\co\Sbb\to\Sett \]
which satisfies the following.
\begin{enumerate}
\item[{\rm (0)}] $M^{\ast}(\xg)=M_{\ast}(\xg)$ for any 0-cell $\xg\in\Sbb^0$. We denote this simply by $M(\xg)$.
\item[{\rm (1)}] [Additivity] For any pair of 0-cells $\xg$ and $\yh$ in $\Sbb$, if we take their 2-coproduct
\[ \xg\ov{\ups_X}{\lra}\xg\am\yh\ov{\ups_Y}{\lla}\yh \]
in $\Sbb$, then the natural map
\begin{equation}\label{RCoeffAdd1}
(M^{\ast}(\ups_X),M^{\ast}(\ups_Y))\co M(\xg\am\yh)\to M(\xg)\times M(\yh)
\end{equation}
is bijective. Also, $M(\emptyset)$ is a singleton.
\item[{\rm (2)}] [Mackey condition] For any 2-fibered product
\begin{equation}\label{RCoeffAdd2}
\xy
(-8,6)*+{\wl}="0";
(8,6)*+{\yh}="2";
(-8,-6)*+{\xg}="4";
(8,-6)*+{\zk}="6";
{\ar^{\delta} "0";"2"};
{\ar_{\gamma} "0";"4"};
{\ar^{\be} "2";"6"};
{\ar_{\al} "4";"6"};
{\ar@{=>}^{\kappa} (-2,0);(2,0)};
\endxy
\end{equation}
in $\Sbb$, the following diagram in $\Sett$ becomes commutative.
\begin{equation}\label{RCoeffAdd3}
\xy
(-12,7)*+{M(\wl)}="0";
(12,7)*+{M(\yh)}="2";
(-12,-7)*+{M(\xg)}="4";
(12,-7)*+{M(\zk)}="6";
{\ar_{M^{\ast}(\delta)} "2";"0"};
{\ar_{M_{\ast}(\gamma)} "0";"4"};
{\ar^{M_{\ast}(\be)} "2";"6"};
{\ar^{M^{\ast}(\al)} "6";"4"};
{\ar@{}|\circlearrowright "0";"6"};
\endxy
\end{equation}
\end{enumerate}
A {\it morphism} $\varphi\co M\to N$ of semi-Mackey functors is a family of maps
\[ \varphi=\{ \varphi_{\xg}\co M(\xg)\to N(\xg) \}_{\xg\in\Sbb^0} \]
compatible with contravariant and covariant parts. Namely, it gives natural transformations
\[ \varphi\co M^{\ast}\tc N^{\ast}\ \ \ \text{and}\ \ \ \varphi\co M_{\ast}\tc N_{\ast}. \]
With the usual composition of natural transformations, we obtain the category of semi-Mackey functors denoted by $\SMackS$.
\end{dfn}

\begin{rem}\label{RemSemiMackMon}
Similarly as in the case of ordinary semi-Mackey functors, the functors  $M^{\ast}$ and $M_{\ast}$ become functors to $\Mon$, and $\varphi$ becomes a natural transformation between such functors.
\end{rem}

\begin{dfn}\label{DefRLinearMack}
An $R$-{\it linear Mackey functor} $M=(M^{\ast},M_{\ast})$ on $\Sbb$ is a pair of a contravariant 2-functor
\[ M^{\ast}\co\Sbb\to\RMod \]
and a covariant 2-functor
\[ M_{\ast}\co\Sbb\to\RMod, \]
which satisfies the following.
\begin{enumerate}
\item[{\rm (0)}] $M^{\ast}(\xg)=M_{\ast}(\xg)\, (=M(\xg))$ for any 0-cell $\xg\in\Sbb^0$. 
\item[{\rm (1)}] [Additivity] For any pair of 0-cells $\xg$ and $\yh$ in $\Sbb$, the natural map $(\ref{RCoeffAdd1})$ is an isomorphism in $\RMod$. $M(\emptyset)=0$ is the zero module.
\item[{\rm (2)}] [Mackey condition] For any 2-fibered product $(\ref{RCoeffAdd2})$ in $\Sbb$, the diagram $(\ref{RCoeffAdd3})$ is a commutative diagram in $\RMod$.
\end{enumerate}
A {\it morphism} $\varphi\co M\to N$ of $R$-linear Mackey functors is a family $\varphi=\{ \varphi_{\xg}\}_{\xg\in\Sbb^0}$ of $R$-homomorphisms compatible with contravariant and covariant parts.
We denote the category of $R$-linear Mackey functors by $\MackSR$.

When $R=\Zbb$, an $R$-linear Mackey functor is simply called a {\it Mackey functor}, and the category of Mackey functors is denoted by $\MackS$.
\end{dfn}

\begin{rem}\label{RemTensorR}
Remark that the additive completion of monoids gives a functor $K_0\co \Mon\to\Ab$. From any semi-Mackey functor $M=(M^{\ast}, M_{\ast})$, by composing $K_0$ we obtain a Mackey functor $K_0M=(K_0\ci M^{\ast},K_0\ci M_{\ast})$ on $\Sbb$.
This gives a functor $K_0\co\SMackS\to\MackS$, which is left adjoint to the inclusion functor $\MackS\hookrightarrow\SMackS$.

Furthermore, tensoring $R$ gives an additive functor $-\otimes_{\mathbb{Z}} R\co \Ab\to\RMod$. From any semi-Mackey functor $M=(M^{\ast}, M_{\ast})$, by composing $-\otimes_{\mathbb{Z}} R$ and $K_0$, we obtain an $R$-linear Mackey functor $M^R=((-\otimes_{\mathbb{Z}} R)\ci K_0\ci M^{\ast},(-\otimes_{\mathbb{Z}} R)\ci K_0\ci M_{\ast})$ on $\Sbb$.
This gives a functor $(-)^R\co\SMackS\to\MackSR$, which is left adjoint to the forgetful functor $\MackSR\rightarrow\SMackS$.
\end{rem}

\begin{dfn}\label{DefDeflMack}
A semi-Mackey functor $M$ on $\Sbb$ is called {\it deflative} if for any stab-surjective 1-cell $\al\co\xg\to\yh$ in $\Sbb$, the equality
\[ M_{\ast}(\al)\ci M^{\ast}(\al)=\id_{M(\yh)} \]
is satisfied. An $R$-linear Mackey functor is called deflative if it is deflative as a semi-Mackey functor. 

The full subcategory of deflative semi-Mackey functors is denoted by $\SMack_{\dfl}(\Sbb)\subseteq\SMackS$. 
Similarly, the full subcategory of deflative $R$-linear Mackey functors is denoted by $\Mack_{\dfl}^R(\Sbb)\subseteq\MackSR$.
\end{dfn}

\begin{ex}
The correspondence $\xg\mapsto\Obig(\xg)$ gives a Mackey functor, which is {\it not} deflative. This is called the {\it bigger Burnside functor}.
\end{ex}

\begin{prop}\label{PropDeflMack}
For an $R$-linear $($resp. semi-$)$Mackey functor $M$ on $\Sbb$, the following are equivalent.
\begin{enumerate}
\item $M$ is deflative.
\item For any surjective group homomorphism $f\co G\to H$, the equality
\[ M_{\ast}(\frac{\pt}{p})\ci M^{\ast}(\frac{\pt}{p})=\id \]
is satisfied for the 1-cell $\frac{\pt}{p}\co\ptg\to\pth$.
\end{enumerate}
\end{prop}

An $R$-linear biset functor $B$ is defined to be an $R$-linear functor $B\co \Bcal_R\to\RMod$, from the {\it biset category} $\Bcal_R$ to $\RMod$.
The biset category which we deal with in this article is the following one.
\begin{introdfn}
An $R$-linear category $\Bcal_R$ is defined as follows.
\begin{enumerate}
\item An object in $\Bcal_R$ is a finite group.
\item For objects $G,H$ in $\Bcal_R$, consider a set of the isomorphism classes of finite $H$-$G$-bisets. This forms a commutative monoid with addition $\am$ and unit $\emptyset$, and thus we can take its additive completion $\Bcal(G,H)$. We define $\Bcal_{R}(G,H)$ by $\Bcal_{R}(G,H)=\Bcal(G,H)\otimes R$. This is the set of morphisms from $G$ to $H$ in $\Bcal_R$.

An $H$-$G$-biset $U$ is written as ${}_HU_G$.
The composition of two consecutive bisets ${}_{H}U_G$ and ${}_{K}V_H$ is given by
\[ V\times_HU=(V\times U)/\sim, \]
where the equivalence relation is defined as
\begin{itemize}
\item[-] $(v,u),(v\ppr,u\ppr)\in V\times U$ are equivalent if there exists $h\in H$ satisfying $v=v\ppr h$ and $u\ppr=hu$.
\end{itemize}
This defines the composition of morphisms in $\Bcal_R$, by linearity.
\end{enumerate}
When $R=\Zbb$, we denote $\Bcal_{\Zbb}$ simply by $\Bcal$. This is a preadditive category.
\end{introdfn}
We denote the category of $R$-linear biset functors by $\BisetFtr^R$. This is naturally equivalent to the category $\Add(\Bcal,\RMod)$ of additive functors from $\Bcal$ to $\RMod$.
The following has been shown in \cite{N_BisetMackey}.
\begin{thm}\label{ThmBF}
There is an equivalence of categories
\[ \Mack_{\dfl}^R(\Sbb)\simeq\BisetFtr^R. \]
\end{thm}

\section{Functor $\MackS\to\BisetFtr$.}

\begin{dfn}\label{DefMtObj}
For any $R$-linear Mackey functor $M$ on $\Sbb$ and any 0-cell $\xg\in\Sbb^0$, we define $J_M(\xg)$ to be the $R$-submodule of $M(\xg)$ generated by the subset
\begin{equation}\label{Eq_elemJ}
\Set{ M\sas(\afr\ppr)M\uas(s)(a)-M\sas(\afr)(a)\, |\, %
\xy
(-10,6)*+{\akp}="0";
(10,6)*+{\ak}="2";
(0,-8)*+{\xg}="4";
{\ar^{s} "0";"2"};
{\ar_{\afr\ppr} "0";"4"};
{\ar^{\afr} "2";"4"};
{\ar@{=>} (3,3);(-2,-1)};
\endxy,
\ \ \begin{array}{l}
a\in M(\ak),\\
s \text{: stab-surjective}
\end{array}
}.
\end{equation}
Denote the quotient module by
\[ \wt{M}(\xg)=M(\xg)/J_M(\xg), \]
and denote the residue homomorphism by
\[ q^{(M)}_{\xg}\co M(\xg)\to\wt{M}(\xg). \]
\end{dfn}

\begin{prop}\label{PropMtMorph}
Let $M$ be an $R$-linear Mackey functor on $\Sbb$. For any 1-cell $\al\co\xg\to\yh$ in $\Sbb$, the homomorphisms $M\uas(\al)$ and $M\sas(\al)$ induce homomorphisms
\[ \wt{M}\uas(\al)\co\wt{M}(\yh)\to\wt{M}(\xg)\quad \text{and}\quad %
\wt{M}\sas(\al)\co\wt{M}(\xg)\to\wt{M}(\yh), \]
which make the following diagrams commutative.
\[
\xy
(-11,7)*+{M(\yh)}="0";
(11,7)*+{M(\xg)}="2";
(-11,-7)*+{\wt{M}(\yh)}="4";
(11,-7)*+{\wt{M}(\xg)}="6";
{\ar^{M\uas(\al)} "0";"2"};
{\ar_{q^{(M)}_{\yh}} "0";"4"};
{\ar^{q^{(M)}_{\xg}} "2";"6"};
{\ar_{\wt{M}\uas(\al)} "4";"6"};
{\ar@{}|\circlearrowright "0";"6"};
\endxy\quad,\quad
\xy
(-11,7)*+{M(\xg)}="0";
(11,7)*+{M(\yh)}="2";
(-11,-7)*+{\wt{M}(\xg)}="4";
(11,-7)*+{\wt{M}(\yh)}="6";
{\ar^{M\sas(\al)} "0";"2"};
{\ar_{q^{(M)}_{\xg}} "0";"4"};
{\ar^{q^{(M)}_{\yh}} "2";"6"};
{\ar_{\wt{M}\sas(\al)} "4";"6"};
{\ar@{}|\circlearrowright "0";"6"};
\endxy
\]
Remark that $\wt{M}\uas(\al)$ and $\wt{M}\sas(\al)$ are uniquely determined by this commutativity, since $q^{(M)}_{\xg}$ and $q^{(M)}_{\yh}$ are surjective.
\end{prop}
\begin{proof}
To show the existence of $\wt{M}\sas(\al)$, it suffices to show $M\sas(\al)(J_M(\xg))\subseteq J_M(\yh)$.
Take any generator $\xi$ of $J_M(\xg)$ of the form
\[ \xi=M\sas(\afr\ppr)M\uas(s)(a)-M\sas(\afr)(a) \]
as in $(\ref{Eq_elemJ})$. 
Then $M\sas(\al)(\xi)\in J_M(\yh)$ follows from
\[ M\sas(\al)(\xi)=M\sas(\al\ci\afr\ppr)M\uas(s)(a)-M\sas(\al\ci\afr)(a)\in J_M(\yh),\]
since there is a diagram in $\Sbb$ as follows.
\[
\xy
(-10,6)*+{\akp}="0";
(10,6)*+{\ak}="2";
(0,-8)*+{\xg}="4";
{\ar^{s} "0";"2"};
{\ar_{\al\ci\afr\ppr} "0";"4"};
{\ar^{\al\ci\afr} "2";"4"};
{\ar@{=>} (3,3);(-2,-1)};
\endxy
\]

The existence of $\wt{M}\uas(\al)$ can be shown in a similar way. Remark that for any diagram
\[
\xy
(-10,6)*+{\blp}="0";
(10,6)*+{\bl}="2";
(0,-8)*+{\yh}="4";
{\ar^{t} "0";"2"};
{\ar_{\bfr\ppr} "0";"4"};
{\ar^{\bfr} "2";"4"};
{\ar@{=>} (3,3);(-2,-1)};
\endxy
\]
in $\Sbb$, if we take 2-fibered products
\[
\xy
(-9,6)*+{\ak}="0";
(9,6)*+{\bl}="2";
(-9,-6)*+{\xg}="4";
(9,-6)*+{\yh}="6";
{\ar^{\gamma} "0";"2"};
{\ar_{\afr} "0";"4"};
{\ar^{\bfr} "2";"6"};
{\ar_{\al} "4";"6"};
{\ar@{=>} (-2,0);(2,0)};
\endxy
\quad \text{and}\quad
\xy
(-9,6)*+{\akp}="0";
(9,6)*+{\blp}="2";
(-9,-6)*+{\xg}="4";
(9,-6)*+{\yh}="6";
{\ar^{\gamma\ppr} "0";"2"};
{\ar_{\afr\ppr} "0";"4"};
{\ar^{\bfr\ppr} "2";"6"};
{\ar_{\al} "4";"6"};
{\ar@{=>} (-2,0);(2,0)};
\endxy,
\]
then by the universality there exist a 1-cell $s\co\ak\to\akp$ and 2-cells which fit into the following diagrams.
\begin{equation}\label{Diag_2Pulls}
\xy
(-10,6)*+{\akp}="0";
(10,6)*+{\ak}="2";
(0,-8)*+{\xg}="4";
{\ar^{s} "0";"2"};
{\ar_{\afr\ppr} "0";"4"};
{\ar^{\afr} "2";"4"};
{\ar@{=>} (3,3);(-2,-1)};
\endxy
\quad ,\quad 
\xy
(-9,7)*+{\akp}="0";
(9,7)*+{\blp}="2";
(-9,-7)*+{\ak}="4";
(9,-7)*+{\bl}="6";
{\ar^{\gamma\ppr} "0";"2"};
{\ar_{s} "0";"4"};
{\ar^{t} "2";"6"};
{\ar_{\gamma} "4";"6"};
{\ar@{=>} (-2,0);(2,0)};
\endxy.
\end{equation}
Moreover, the right diagram in $(\ref{Diag_2Pulls})$ becomes a 2-fibered product.

Since stab-surjectivity is stable under 2-pullbacks, $s$ becomes stab-surjective when $t$ is so. In this case, for any element $b\in M(\bl)$, we have
\begin{eqnarray*}
M\uas(\al)\!\!\!\!\!\!\!\! &&\!\!\!\!\!\!\!\! (M\sas(\bfr\ppr)M\uas(t)(b)-M\sas(\bfr)(b))\\
&=& M\sas(\afr\ppr)M\uas(\gamma\ppr)M\uas(t)(b)-M\sas(\afr)M\uas(\gamma)(b)\\
&=& M\sas(\afr\ppr)M\uas(s)\big(M\uas(\gamma)(b)\big)-M\sas(\afr)\big(M\uas(\gamma)(b)\big)\ \in\ J_M(\xg),
\end{eqnarray*}
which means $M\uas(\al)(J_M(\yh))\subseteq J_M(\xg)$.
\end{proof}

\begin{cor}\label{CorMtMack}
Let $M$ be any $R$-linear Mackey functor on $\Sbb$. With the structure maps obtained in Proposition \ref{PropMtMorph}, the pair $\wt{M}=(\wt{M}\uas,\wt{M}\sas)$ becomes an $R$-linear deflative Mackey functor on $\Sbb$, and $q_M=\{ q^{(M)}_{\xg} \}_{\xg\in\Sbb^0}$ forms a morphism $q\co M\to\wt{M}$.
\end{cor}
\begin{proof}
From Proposition \ref{PropMtMorph}, it immediately follows that $\wt{M}$ is an $R$-linear Mackey functor on $\Sbb$, and $q$ is a morphism.

We show the deflativity of $\wt{M}$. For any stab-surjective 1-cell $s\co\xgp\to\xg$, if we draw a diagram
\[
\xy
(-10,6)*+{\xgp}="0";
(10,6)*+{\xg}="2";
(0,-8)*+{\xg}="4";
{\ar^{s} "0";"2"};
{\ar_{s} "0";"4"};
{\ar^{\id} "2";"4"};
{\ar@{}|\circlearrowright (3,3);(-3,-1)};
\endxy,
\]
we see that for any $m\in M(\xg)$, the element
\[ M\sas(s)M\uas(s)(m)-m \]
belongs to $J_M(\xg)$. This means that the equation $\wt{M}\sas(s)\wt{M}\uas(s)=\id$ holds for $\wt{M}$.
\end{proof}

\begin{prop}\label{PropMtAdj}
Let $M$ be any $R$-linear Mackey functor on $\Sbb$. The morphism $q_M\co M\to\wt{M}$ induces a bijection 
\[ -\ci q_M\co\MackdSR(\wt{M},N)\ov{\cong}{\lra}\MackSR(M,N) \]
for any deflative $R$-linear Mackey functor $N$ on $\Sbb$.
\end{prop}
\begin{proof}
Once the following claim is shown, the rest will follow from the surjectivity of $q^{(M)}_{\xg}$ for any $\xg\in\Sbb^0$.
\begin{claim}\label{ClaimMtAdj}
Let $M,N$ be as above. Then for any $f\in\MackSR(M,N)$ and any $\xg\in\Sbb^0$, the homomorphism $f_{\xg}\co M(\xg)\to N(\xg)$ factors through to yield a homomorphism $\wt{f}_{\xg}\co \wt{M}(\xg)\to N(\xg)$ as in the following diagram.
\[
\xy
(-12,4)*+{M(\xg)}="0";
(12,4)*+{N(\xg)}="2";
(-4,-11)*+{\wt{M}(\xg)}="4";
{\ar^{f_{\xg}} "0";"2"};
{\ar_{q^{(M)}_{\xg}} "0";"4"};
{\ar_{\wt{f}_{\xg}} "4";"2"};
{\ar@{}|\circlearrowright (-3,4);(-1,-8)};
\endxy
\]
\end{claim}


\begin{proof}[Proof of Claim \ref{ClaimMtAdj}]
It suffices to show $f_{\xg}(J_M(\xg))=0$. Take any generator
\[ \xi=M\sas(\afr\ppr)M\uas(s)(a)-M\sas(\afr)(a) \]
as in $(\ref{Eq_elemJ})$. Then by the deflativity of $N$, indeed we obtain
\begin{eqnarray*}
f_{\xg}(\xi)&=&f_{\xg}M\sas(\afr\ppr)M\uas(s)(a)-f_{\xg}M\sas(\afr)(a)\\
&=&N\sas(\afr\ppr)N\uas(s)(f_{\ak}(a))-N\sas(\afr)(f_{\ak}(a))\\
&=&N\sas(\afr\ppr)N\uas(s)(f_{\ak}(a))-N\sas(\afr)N\sas(s)N\uas(s)(f_{\ak}(a))\\
&=&0.
\end{eqnarray*}
\end{proof}
\end{proof}

\begin{cor}\label{CorMtAdj}
The correspondence $M\mapsto\wt{M}$ gives a functor
\[ \wt{(-)}\co\MackSR\to\MackdSR \]
which is left adjoint to the inclusion $\MackdSR\hookrightarrow\MackSR$.

Moreover $q=\{q_M\}_{M\in\Ob(\MackSR)}$ gives a natural transformation
\[ q\co \Id\tc \wt{(-)}, \]
which gives an isomorphism $q_M\co M\ov{\cong}{\lra}\wt{M}$ whenever $M$ is deflative.
\end{cor}
\begin{proof}
As in the proof of Proposition \ref{PropMtAdj}, for any morphism $f\co M\to M\ppr$ in $\MackSR$, there is a unique morphism $\wt{f}\co\wt{M}\to\wt{M}\ppr$ in $\MackdSR$ which makes the following diagram commutative.
\[
\xy
(-7,6)*+{M}="0";
(7,6)*+{M\ppr}="2";
(-7,-6)*+{\wt{M}}="4";
(7,-6)*+{\wt{M}\ppr}="6";
{\ar^{f} "0";"2"};
{\ar_{q_M} "0";"4"};
{\ar^{q_{M\ppr}} "2";"6"};
{\ar_{\wt{f}} "4";"6"};
{\ar@{}|\circlearrowright "0";"6"};
\endxy
\]
By the uniqueness, it can be easily shown that this correspondence gives a functor $\wt{(-)}\co\MackSR\to\MackdSR$. Adjointness also follows from Proposition \ref{PropMtAdj}. The latter part is trivial.
\end{proof}

\begin{ex}
For the bigger Burnside functor $\Obig\in\Ob(\MackS)$, we have $\wt{\Obig}\cong\Omega$, where $\Omega$ denotes the ordinary Burnside functor.
\end{ex}


\section{Analog of Jacobson's $F$-Burnside construction.}

\begin{dfn}\label{DefSaddC}
Define category $\AddCR$ (resp. $\SaddC$) as follows.
\begin{itemize}
\item[-] An object in $\AddCR$ (resp. $\SaddC$) is a contravariant functor
\[ F\co\Csc\to\RMod \quad (\text{resp.}\ \  E\co\Csc\to\Sett) \]
which sends coproducts in $\Csc$ to products in $\RMod$ (resp. $\Sett$).
\item[-] A morphism is a natural transformation.
\end{itemize}
\end{dfn}


\begin{dfn}\label{DefFS2cat}
Let $E$ be any object in $\SaddC$. For any $\xg\in\Sbb^0$, define a 2-category $\ESxg$ as follows.
\begin{enumerate}
\item[{\rm (0)}] A 0-cell is a pair $(\akaxg,\xi)$ of a 0-cell $(\akaxg)$ in $\Sxg$ and an element $\xi\in E(\ak)$.
\item[{\rm (1)}] A 1-cell $(\vp,\mu)\co(\akaxg,\xi)\to(\akaxgp,\xi\ppr)$ is a 1-cell $(\vp,\mu)$ in $\Sxg$, which satisfies
\[ E(\vp)(\xi\ppr)=\xi. \]
\item[{\rm (2)}] A 2-cell $\ep\co(\vp,\mu)\tc(\vp\ppr,\mu\ppr)$ is a 2-cell $\ep$ in $\Sxg$.
\end{enumerate}
Compositions are induced from those in $\Sxg$.
\end{dfn}


\begin{dfn}\label{DefAEachGX}
Let $E,\xg$ and $\ESxg$ be as in Definition \ref{DefFS2cat}. We denote the set of adjoint equivalence classes of 0-cells in $\ESxg$ by $\Afr[E](\xg)$. For each 0-cell $(\akaxg,\xi)\in(\ESxg)^0$, we denote its adjoint equivalence class by
\[ [\akaxg,\xi]\in\Afr [E](\xg). \]

For each pair of elements $[\akaxg,\xi]$ and $[\akaxgp,\xi\ppr]$, we define their {\it sum} by
\begin{equation}\label{Eq_Sum_axi}
[\akaxg,\xi]+[\akaxgp,\xi\ppr]=[\ak\am\akp\ov{\afr\cup\afr\ppr}{\lra}\xg,\xi\sika\xi\ppr],
\end{equation}
where
\begin{itemize}
\item[-] $\ak\am\akp\ov{\afr\cup\afr\ppr}{\lra}\xg$ is the 2-coproduct in $\Sxg$,
\item[-] $\xi\sika\xi\ppr$ is the image of $(\xi,\xi\ppr)\in E(\ak)\times E(\akp)$ under the natural isomorphism
\[ E(\ak)\times E(\akp)\ov{\cong}{\lra}E(\ak\am\akp). \]
\end{itemize}
\end{dfn}

\begin{rem}\label{RemAEachGX}
The definition of the sum $(\ref{Eq_Sum_axi})$ does not depend on the choice of representatives.
\end{rem}


\begin{dfn}\label{DefAEachalpha}
Let $E$ be any object in $\SaddC$. For any 1-cell $\al\co\xg\to\yh$ in $\Sbb$, we define homomorphisms
\begin{eqnarray*}
&\Afr [E]\sas(\al)\co \Afr [E](\xg)\to \Afr [E](\yh),&\\
&\Afr [E]\uas(\al)\co \Afr [E](\yh)\to \Afr [E](\xg)\,&
\end{eqnarray*}
by the following.
\begin{enumerate}
\item $\Afr [E]\sas(\al)([\akaxg,\xi])=[\ak\ov{\al\ci\afr}{\lra}\yh,\xi]$ for any $[\akaxg,\xi]\in\Afr [E](\xg)$.
\item $\Afr [E]\uas(\al)([\blbyh,\eta])=[\akaxg,E(\gamma)(\eta)]$ for any $[\blbyh,\eta]\in\Afr [E](\yh)$, where
\[
\xy
(-8,6)*+{\ak}="0";
(8,6)*+{\bl}="2";
(-8,-6)*+{\xg}="4";
(8,-6)*+{\yh}="6";
{\ar^{\gamma} "0";"2"};
{\ar_{\afr} "0";"4"};
{\ar^{\bfr} "2";"6"};
{\ar_{\al} "4";"6"};
{\ar@{=>} (-2,0);(2,0)};
\endxy
\]
is a 2-fibered product in $\Sbb$.
\end{enumerate}
\end{dfn}


\begin{rem}
In Definition \ref{DefAEachalpha}, those $\Afr [E]\sas(\al)$ and $\Afr [E]\uas(\al)$ are indeed well-defined monoid homomorphisms. Moreover, if 1-cells $\al,\al\ppr\co\xg\to\yh$ satisfy $\und{\al}=\und{\al}\ppr$ in $\Csc$, then we have
\[ \Afr [E]\sas(\al)=\Afr [E]\sas(\al\ppr)\ \ \text{and}\ \ \Afr [E]\uas(\al)=\Afr [E]\uas(\al\ppr). \]
These correspondences form a contravariant functor
\[ \Afr [E]\sas\co\Csc\to\RMod \]
and a covariant functor
\[ \Afr [E]\uas\co\Csc\to\RMod. \]
\end{rem}
\begin{proof}
This follows from the universality of 2-coproducts and 2-fibered products in $\Sbb$.
\end{proof}


\begin{prop}\label{PropAFEachF}
For any $E\in\Ob(\SaddC)$, the pair
\[ \Afr [E]=(\Afr [E]\uas,\Afr [E]\sas) \]
becomes a semi-Mackey functor on $\Sbb$.
\end{prop}
\begin{proof}
\noindent {\bf [Additivity]}

Let $\xg\ov{\ups_{\xg}}{\lra}\xg\am\yh\ov{\ups_{\yh}}{\lla}\yh$ be a 2-coproduct in $\Sbb$. Then
\[ \big(\Afr [E]\uas(\ups_{\xg}),\Afr [E]\uas(\ups_{\yh})\big)\co \Afr [E](\xg\am\yh)\to \Afr [E](\xg)\times \Afr [E](\yh) \]
becomes an isomorphism. In fact,
\[
\xy
(-24,4)*+{\Afr [E](\xg)\times \Afr [E](\yh)}="0";
(24,4)*+{\Afr [E](\xg\am\yh)}="2";
(-24,0)*+{\rotatebox{90}{\text{$\in$}}}="10";
(24,0)*+{\rotatebox{90}{\text{$\in$}}}="12";
(-24,-4)*+{([\akaxg,\xi],[\blbyh,\eta])}="20";
(27,-4)*+{[\ak\am\bl\ov{\afr\am\bfr}{\lra}\xg\am\yh,\xi\sika\eta]}="22";
{\ar "0";"2"};
{\ar@{|->} "20";"22"};
\endxy
\]
gives the inverse.

\bigskip

\noindent {\bf [Mackey condition]}

Let
\[
\xy
(-8,6)*+{\yh}="0";
(8,6)*+{\xg}="2";
(-8,-6)*+{\yhp}="4";
(8,-6)*+{\xgp}="6";
{\ar^{\al} "0";"2"};
{\ar_{\lam} "0";"4"};
{\ar^{\rho} "2";"6"};
{\ar_{\al\ppr} "4";"6"};
{\ar@{=>} (-2,0);(2,0)};
\endxy
\]
be any 2-fibered product in $\Sbb$. For any 1-cell $\akaxg$ in $\Sbb$, if we take a 2-fibered product
\[
\xy
(-8,6)*+{\bl}="0";
(8,6)*+{\ak}="2";
(-8,-6)*+{\yh}="4";
(8,-6)*+{\xg}="6";
{\ar^{\gamma} "0";"2"};
{\ar_{\bfr} "0";"4"};
{\ar^{\afr} "2";"6"};
{\ar_{\al} "4";"6"};
{\ar@{=>} (-2,0);(2,0)};
\endxy
\]
in $\Sbb$, then
\[
\xy
(-8,6)*+{\bl}="0";
(8,6)*+{\ak}="2";
(-8,-6)*+{\yhp}="4";
(8,-6)*+{\xgp}="6";
{\ar^{\gamma} "0";"2"};
{\ar_{\lam\ci\bfr} "0";"4"};
{\ar^{\rho\ci\afr} "2";"6"};
{\ar_{\al\ppr} "4";"6"};
{\ar@{=>} (-2,0);(2,0)};
\endxy
\]
also becomes a 2-fibered product, and thus for any $[\akaxg,\xi]\in\Afr [E](\xg)$, we have
\begin{eqnarray*}
\Afr [E]\uas(\al\ppr)\ci\Afr [E]\sas(\rho)([\akaxg,\xi])%
&=&\Afr [E]\uas(\al\ppr)([\ak\ov{\rho\ci\afr}{\lra}\xgp,\xi])\\
&=&[\bl\ov{\lam\ci\bfr}{\lra}\yhp,E(\gamma)(\xi)]\\
&=&\Afr [E]\sas(\lam)([\blbyh,E(\gamma)(\xi)])\\
&=&\Afr [E]\sas(\lam)\ci\Afr [E]\uas(\al)([\akaxg,\xi]).
\end{eqnarray*}
\end{proof}


\begin{prop}\label{PropAFunct}
The correspondence
\[ \Afr [-]\co \SaddC\to\SMackS\ ; \ E\mapsto \Afr [E] \]
forms a functor.
\end{prop}
\begin{proof}
If $\tau\co E\to E\ppr$ is a morphism in $\SaddC$, then $\Afr [\tau]\co \Afr [E]\to\Afr [E\ppr]$ is defined by
\[
\xy
(-32.4,3.8)*+{\Afr [\tau]_{\xg}\co}="-2";
(-20,4)*+{\Afr [E](\xg)}="0";
(20,4)*+{\Afr [E\ppr](\xg)}="2";
(-20,0)*+{\rotatebox{90}{\text{$\in$}}}="10";
(20,0)*+{\rotatebox{90}{\text{$\in$}}}="12";
(-20,-4)*+{[\akaxg,\xi]}="20";
(20,-4)*+{[\akaxg,\tau_{\ak}(\xi)]}="22";
{\ar "0";"2"};
{\ar@{|->} "20";"22"};
\endxy
\]
for each $\xg\in\Sbb^0$. Then for any 1-cell $\al\co\xg\to\yh$, the commutativity of
\[
\xy
(-13,7)*+{\Afr [E](\xg)}="0";
(13,7)*+{\Afr [E\ppr](\xg)}="2";
(-13,-7)*+{\Afr [E](\yh)}="4";
(13,-7)*+{\Afr [E\ppr](\yh)}="6";
{\ar^{\Afr [\tau]_{\xg}} "0";"2"};
{\ar_{\Afr [E]\sas(\al)} "0";"4"};
{\ar^{\Afr [E\ppr]\sas(\al)} "2";"6"};
{\ar_{\Afr [\tau]_{\yh}} "4";"6"};
{\ar@{}|\circlearrowright "0";"6"};
\endxy
\quad\text{and}\quad
\xy
(-13,7)*+{\Afr [E](\yh)}="0";
(13,7)*+{\Afr [E\ppr](\yh)}="2";
(-13,-7)*+{\Afr [E](\xg)}="4";
(13,-7)*+{\Afr [E\ppr](\xg)}="6";
{\ar^{\Afr [\tau]_{\yh}} "0";"2"};
{\ar_{\Afr [E]\uas(\al)} "0";"4"};
{\ar^{\Afr [E\ppr]\uas(\al)} "2";"6"};
{\ar_{\Afr [\tau]_{\xg}} "4";"6"};
{\ar@{}|\circlearrowright "0";"6"};
\endxy
\]
is easily confirmed. It is also easy to show that this gives a functor $\Afr [-]\co \SaddC\to\SMackS$.
\end{proof}


\begin{prop}\label{PropAAdj}
The functor $\Afr [-]\co\SaddC\to\SMackS$ obtained in Proposition \ref{PropAFunct} is left adjoint to the forgetful functor
\[ \SMackS\to\SaddC\ ;\ M=(M\uas,M\sas)\mapsto M\uas. \]
\end{prop}
\begin{proof}
For any $M\in\Ob(\SMackS)$ and $E\in\Ob(\SaddC)$, we construct natural maps
\[ \Th\co\SMackS(\Afr [E],M)\to\SaddC(E,M\uas) \]
and
\[ \Phi\co\SaddC(E,M\uas)\to\SMackS(\Afr [E],M), \]
which are inverse to each other.


\bigskip

\noindent {\bf {\rm (1)} Construction of $\Th$.}

For any $\vp\in\SMackS(\Afr [E],M)$ and any $\xg\in\Sbb^0$, we define
\[ \Th(\vp)_{\xg}\co E(\xg)\to M\uas(\xg) \]
by
\[ \Th(\vp)_{\xg}(\xi)=\vp_{\xg}([\xgixg,\xi]) \]
for any $\xi\in E(\xg)$. Then for any 1-cell $\al\co\xg\to\yh$, the following diagram becomes commutative.
\[
\xy
(-12,7)*+{E(\yh)}="0";
(12,7)*+{M\uas(\yh)}="2";
(-12,-7)*+{E(\xg)}="4";
(12,-7)*+{M\uas(\yh)}="6";
{\ar^{\Th(\vp)_{\yh}} "0";"2"};
{\ar_{E(\al)} "0";"4"};
{\ar^{M\uas(\al)} "2";"6"};
{\ar_{\Th(\vp)_{\xg}} "4";"6"};
{\ar@{}|\circlearrowright "0";"6"};
\endxy
\]

Indeed, for any $\eta\in E(\yh)$, we have
\begin{eqnarray*}
M\uas(\al)\Th(\vp)_{\yh}(\eta)&=&M\uas(\al)\vp_{\yh}([\yhiyh,\eta])\\
&=&\vp_{\xg}\Afr [E]\uas(\al)([\yhiyh,\eta])\\
&=&\vp_{\xg}([\xgixg,E(\al)(\eta)])\\
&=&\Th(\vp)_{\xg}E(\al)(\eta).
\end{eqnarray*}


This means that $\Th(\vp)=\{\Th(\vp)_{\xg} \}_{\xg\in\Sbb^0}$ forms a morphism $\Th(\vp)\co E\to M\uas$.

\bigskip

\noindent {\bf {\rm (2)} Construction of $\Phi$.}

For any $\thh\in\SaddC(E,M\uas)$ and $\xg\in\Sbb^0$, we define
\[ \Phi(\thh)_{\xg}\co \Afr [E](\xg)\to M(\xg) \]
by
\[ \Phi(\thh)_{\xg}([\akaxg,\xi])=M\sas(\afr)\thh_{\ak}(\xi). \]

To show that this defines a morphism of Mackey functors, it suffices to show the commutativity of the following diagrams, for any 1-cell $\al\co\xg\to\yh$.
\[
\xy
(-12,7)*+{\Afr [E](\xg)}="0";
(12,7)*+{M(\xg)}="2";
(-12,-7)*+{\Afr [E](\yh)}="4";
(12,-7)*+{M(\yh)}="6";
{\ar^{\Phi(\thh)_{\xg}} "0";"2"};
{\ar_{\Afr [E]\sas(\al)} "0";"4"};
{\ar^{M\sas(\al)} "2";"6"};
{\ar_{\Phi(\thh)_{\yh}} "4";"6"};
{\ar@{}|\circlearrowright "0";"6"};
\endxy
\quad,\quad
\xy
(-12,7)*+{\Afr [E](\yh)}="0";
(12,7)*+{M(\yh)}="2";
(-12,-7)*+{\Afr [E](\xg)}="4";
(12,-7)*+{M(\xg)}="6";
{\ar^{\Phi(\thh)_{\yh}} "0";"2"};
{\ar_{\Afr [E]\uas(\al)} "0";"4"};
{\ar^{M\uas(\al)} "2";"6"};
{\ar_{\Phi(\thh)_{\xg}} "4";"6"};
{\ar@{}|\circlearrowright "0";"6"};
\endxy.
\]

This can be confirmed as follows.
For any $[\akaxg,\xi]\in\Afr [E](\xg)$, we have
\begin{eqnarray*}
M\sas(\al)\Phi(\thh)_{\xg}([\akaxg,\xi])&=&M\sas(\al)M\sas(\afr)\thh_{\ak}(\xi)\\
&=&M\sas(\al\ci\afr)\thh_{\ak}(\xi)\\
&=&\Phi(\thh)_{\yh}\Afr [E]\sas(\al)([\akaxg,\xi]).
\end{eqnarray*}

For any $[\blbyh,\eta]\in\Afr [E](\yh)$, if we take a 2-fibered product
\[
\xy
(-8,6)*+{\ak}="0";
(8,6)*+{\bl}="2";
(-8,-6)*+{\xg}="4";
(8,-6)*+{\yh}="6";
{\ar^{\gamma} "0";"2"};
{\ar_{\afr} "0";"4"};
{\ar^{\bfr} "2";"6"};
{\ar_{\al} "4";"6"};
{\ar@{=>} (-2,0);(2,0)};
\endxy,
\]
then we have
\begin{eqnarray*}
M\uas(\al)\Phi(\thh)_{\yh}([\blbyh,\eta])&=&M\uas(\al)M\sas(\bfr)\thh_{\bl}(\eta)\\
&=&M\sas(\afr)M\uas(\gamma)\thh_{\bl}(\eta)\\
&=&M\sas(\afr)\thh_{\ak}E(\eta)\\
&=&\Phi(\thh)_{\xg}([\akaxg,E(\eta)])\\
&=&\Phi(\thh)_{\xg}\Afr [E]\sas(\al)([\blbyh,\eta]).
\end{eqnarray*}


\bigskip

\noindent {\bf {\rm (3)} $\Th\ci\Phi=\id$.}

For any $\thh\in\SaddC(E,M\uas)$ and $\xg\in\Sbb^0$, we have
\begin{eqnarray*}
(\Th\ci\Phi(\thh))_{\xg}(\xi)&=&\Phi(\thh)_{\xg}([\xgixg,\xi])\\
&=&M\sas(\id)\thh_{\xg}(\xi)\\
&=&\thh_{\xg}(\xi)\quad(\fa\xi\in E(\xg)),
\end{eqnarray*}
which means $\Th\ci\Phi(\thh)=\thh$.

\bigskip

\noindent {\bf {\rm (4)} $\Phi\ci\Th=\id$.}

For any $\vp\in\SMackS(\Afr [E],M)$ and $\xg\in\Sbb^0$, we have
\begin{eqnarray*}
(\Phi\ci\Th(\vp))_{\xg}([\akaxg,\xi])&=&M\sas(\afr)\Th(\varphi)_{\ak}(\xi)\\
&=&M\sas(\afr)\varphi_{\ak}([\akiak,\xi])\\
&=&\vp_{\xg}\Afr [E]\sas(\afr)([\akiak,\xi])\\
&=&\vp_{\xg}([\akaxg,\xi]),
\end{eqnarray*}
which means $\Phi\ci\Th(\vp)=\vp$.
\end{proof}


\begin{dfn}\label{DefFBurn}
Define a functor
\[ \Obig^R[-]\co\SaddC\to\MackSR \]
to be the composition of
\[ \SaddC\ov{\Afr [-]}{\lra}\SMackS\ov{(-)^R}{\lra}\MackSR. \]
When $R=\Zbb$, we simply write this functor as $\Obig [-]$.
\end{dfn}

\begin{cor}\label{CorFBurn}
This gives a left adjoint of the forgetful functor
\[ \MackSR\to\SaddC\ ;\ M=(M\uas,M\sas)\mapsto M\uas. \]
\end{cor}
\begin{proof}
This follows from Remark \ref{RemTensorR} and Proposition \ref{PropAAdj}.
\end{proof}


\begin{rem}\label{RemFBurnAdd}
For any $E\in\Ob(\SaddC)$ and any $\xg\in\Sbb^0$, elements in $\Obig^R[E](\xg)$ can be written as $R$-linear combinations of elements in $\Afr [E](\xg)$.
\end{rem}

\begin{rem}\label{RemFBurn}
For any $M\in\Ob(\MackSR)$ and any $E\in\Ob(\SaddC)$, the natural bijections
\[ \MackSR(\Obig^R [E],M)\leftrightarrow\SaddC(E,M\uas) \]
are obtained by composing the natural bijection associated to the adjointness of $(-)^R$
\[ \MackSR(\Obig^R [E],M)\ov{\cong}{\lra}\SMackS(\Afr [E],M) \]
with those $\Phi, \Th$ defined in the proof of Proposition \ref{PropAAdj}.
We denote them by
\[ \ovl{\Th}\co\MackSR(\Obig [E],M)\to\SaddC(E,M\uas) \]
and
\[ \ovl{\Phi}\co\SaddC(E,M\uas)\to\MackSR(\Obig [E],M). \]
\end{rem}

\begin{cor}
By composing functors
\[ \SaddC\ov{\Obig^R [-]}{\lra}\MackSR\ov{\wt{(-)}}{\lra}\MackdSR\ov{\simeq}{\lra}\BisetFtr^R, \]
we obtain a functor $\SaddC\to\BisetFtr^R$.
\end{cor}


\section{Analog of Boltje's $(-)_+$-construction.}

In the last section, we constructed a functor
\[ \Obig^R[-]\co\SaddC\to \MackSR, \]
which is left adjoint to $\MackSR\to\SaddC$. In this section, when $F$ is in $\AddCR$, furthermore we construct another Mackey functor $F_+$ by taking quotient of $\Obig^R[F]$. Remark that an element in $\Obig^R[F](\xg)$ is an $R$-linear combination of elements in $\Afr [F](\xg)$.

\begin{dfn}\label{DefIGX}
Let $F\co\Csc\to\RMod$ be any object in $\AddCR$. For any object $\xg\in\Sbb^0$, define $I_F(\xg)\subseteq\Obig^R[F](\xg)$ to be the $R$-submodule generated by the following subset of $\Obig^R[F](\xg)$.
\begin{equation}\label{Eq_elemI}
\Set{ \begin{array}{l} {[}\akaxg,r_1\xi_1+r_2\xi_2]\\
-\big(r_1[\akaxg,\xi_1]+r_2[\akaxg,\xi_2]\big)
\end{array}
\, |\, 
\begin{array}{c}
\akaxg\in(\Sxg)^0,\\
\xi_1,\xi_2\in F(\ak),\\
r_1,r_2\in R
\end{array}   }
\end{equation}

Denote the quotient module by $F_+(\xg)$, and the residue homomorphism by
\[ d^{(F)}_{\xg}\co\Obig^R[F](\xg)\to F_+(\xg). \]
\end{dfn}


\begin{prop}\label{PropFbMorph}
Let $F$ be an object in $\AddCR$. For any 1-cell $\al\co\xg\to\yh$, the homomorphisms $\Obig^R[F]\uas(\al)$ and $\Obig^R[F]\sas(\al)$ induce homomorphisms
\begin{eqnarray*}
&\ \ F_+\uas(\al)\co F_+(\yh)\to F_+(\xg),&\\
&F_{+\ast}(\al)\co F_+(\xg)\to F_+(\yh),&
\end{eqnarray*}
which make the following diagrams commutative.
\[
\xy
(-16,7)*+{\Obig^R[F](\yh)}="0";
(16,7)*+{\Obig^R[F](\xg)}="2";
(-16,-7)*+{F_+(\yh)}="4";
(16,-7)*+{F_+(\xg)}="6";
{\ar^{\Obig^R[F]\uas(\al)} "0";"2"};
{\ar_{d^{(F)}_{\yh}} "0";"4"};
{\ar^{d^{(F)}_{\xg}} "2";"6"};
{\ar_{F_+\uas(\al)} "4";"6"};
{\ar@{}|\circlearrowright "0";"6"};
\endxy\quad,\quad
\xy
(-16,7)*+{\Obig^R[F](\xg)}="0";
(16,7)*+{\Obig^R[F](\yh)}="2";
(-16,-7)*+{F_+(\xg)}="4";
(16,-7)*+{F_+(\yh)}="6";
{\ar^{\Obig^R[F]\sas(\al)} "0";"2"};
{\ar_{d^{(F)}_{\xg}} "0";"4"};
{\ar^{d^{(F)}_{\yh}} "2";"6"};
{\ar_{F_{+\ast}(\al)} "4";"6"};
{\ar@{}|\circlearrowright "0";"6"};
\endxy.
\]
Remark that $F_+\uas(\al)$ and $F_{+\ast}(\al)$ are uniquely determined by this commutativity, since $d^{(F)}_{\xg}$ and $d^{(F)}_{\yh}$ are surjective.
\end{prop}


\begin{proof}
We show in a similar way as in the proof of Proposition \ref{PropMtMorph}.

\smallskip

\noindent {\bf {\rm (1)} Existence of $F_{+\ast}(\al)$. }

It suffices to show $\Obig^R[F]\sas(\al)(I_F(\xg))\subseteq I_F(\yh)$. Take any generator $\xi$ of $I_F(\xg)$
\[ \xi=[\akaxg,r_1\xi_1+r_2\xi_2]-\big( r_1[\akaxg,\xi_1]+r_2[\akaxg,\xi_2]\big) \]
as in $(\ref{Eq_elemI})$. Then the element
\begin{eqnarray*}
\Obig^R[F]\sas(\al)(\xi)&=&[\ak\ov{\al\ci\afr}{\lra}\yh,r_1\xi_1+r_2\xi_2]\\
&-&\big( r_1[\ak\ov{\al\ci\afr}{\lra}\yh,\xi_1]+r_2[\ak\ov{\al\ci\afr}{\lra}\yh,\xi_2]\big)
\end{eqnarray*}
belongs to $I_F(\yh)$. 

\bigskip

\noindent {\bf {\rm (2)} Existence of $F_+\uas(\al)$.}

For any 1-cell $\blbyh$ in $\Sbb$, let
\[
\xy
(-8,6)*+{\ak}="0";
(8,6)*+{\bl}="2";
(-8,-6)*+{\xg}="4";
(8,-6)*+{\yh}="6";
{\ar^{\gamma} "0";"2"};
{\ar_{\afr} "0";"4"};
{\ar^{\bfr} "2";"6"};
{\ar_{\al} "4";"6"};
{\ar@{=>} (-2,0);(2,0)};
\endxy
\]
be a 2-fibered product.
For any $\eta_1,\eta_2\in F(\bl)$ and any $r_1,r_2\in R$, we have
\begin{eqnarray*}
&&\Obig^R[F]\uas(\al)\bigg([\blbyh,r_1\eta_1+r_2\eta_2]-\big(r_1[\blbyh,\eta_1]+r_2[\blbyh,\eta_2]\big)\bigg)\\
&=&[\akaxg,F(\gamma)(r_1\eta_1+r_2\eta_2)]\\
&\ &\ -\big(r_1[\akaxg,F(\gamma)(\eta_1)]+r_2[\akaxg,F(\gamma)(\eta_2)]\big)\\
&=&[\akaxg,r_1F(\gamma)(\eta_1)+r_2F(\gamma)(\eta_2)]\\
&\ &\ -\big(r_1[\akaxg,F(\gamma)(\eta_1)]+r_2[\akaxg,F(\gamma)(\eta_2)]\big)\\
&\in&I_F(\xg).
\end{eqnarray*}
\end{proof}


\begin{cor}\label{CorFbMack}
Let $F$ be any object in $\AddCR$. With the structure maps obtained in Proposition \ref{PropFbMorph}, the pair $F_+=(F_+\uas,F_{+\ast})$ becomes an $R$-linear Mackey functor on $\Sbb$, and $d_F=\{d^{(F)}_{\xg}\}_{\xg\in\Sbb^0}$ forms a morphism of $R$-linear Mackey functors $d_F\co\Obig^R[F]\to F_+$.
\end{cor}
\begin{proof}
This immediately follows from Proposition \ref{PropFbMorph}.
\end{proof}

\begin{prop}
The correspondence
\[ (-)_+\co\AddCR\to\MackSR\ ;\ F\mapsto F_+ \]
forms a functor.
\end{prop}
\begin{proof}
For any morphism $\tau\co F\to F\ppr$ in $\AddCR$, we can show easily that for each $\xg\in\Sbb^0$, there is a unique module homomorphism
\[ (\tau_+)_{\xg}\co F_+(\xg)\to F\ppr_+(\xg) \]
which makes the following diagram commutative.
\[
\xy
(-18,16)*+{\Afr [F](\xg)}="0";
(18,16)*+{\Afr [F\ppr](\xg)}="2";
(-18,0)*+{\Obig^R[F](\xg)}="4";
(18,0)*+{\Obig^R[F\ppr](\xg)}="6";
(-18,-16)*+{F_+(\xg)}="8";
(18,-16)*+{F\ppr_+(\xg)}="10";
{\ar^{\Afr [\tau]_{\xg}} "0";"2"};
{\ar_{\text{induced by}\ -\otimes_R} "0";"4"};
{\ar^{\text{induced by}\ -\otimes_R} "2";"6"};
{\ar_{\Obig^R[\tau]_{\xg}} "4";"6"};
{\ar_{d_F} "4";"8"};
{\ar^{d_{F\ppr}} "6";"10"};
{\ar_{(\tau_+)_{\xg}} "8";"10"};
{\ar@{}|\circlearrowright "0";"6"};
{\ar@{}|\circlearrowright "4";"10"};
\endxy
\]
This gives a morphism $\tau_+\co F_+\to F_+\ppr$. The functoriality of $(-)_+$ can be checked easily.
\end{proof}


\begin{prop}\label{PropFbAdj}
The functor $(-)_+\co \AddCR\to\MackSR$ is left adjoint to the forgetful functor taking contravariant parts
\[ \MackSR\to\AddCR\ ;\ M=(M\uas,M\sas)\mapsto M\uas. \]
\end{prop}
\begin{proof}
It suffices to construct a natural bijection
\[ \MackSR(F_+,M)\cong\AddCR(F,M\uas) \]
for any $F\in\Ob(\AddCR)$ and any $M\in\Ob(\MackSR)$.

Remark that $d_F\co\Obig^R[F]\to F_+$ induces a injection
\[ d_F^{\sh}=(-\ci d_F)\co \MackSR(F_+,M)\hookrightarrow\MackSR(\Obig^R[F],M), \]
since $d^{(F)}_{\xg}$ is surjective for each $\xg\in\Sbb^0$. An element $\vp\in\MackSR(\Obig^R[F],M)$ comes from $\MackSR(F_+,M)$ if and only if $\vp$ satisfies
\[ \vp_{\xg}(I_F(\xg))=0 \]
for any $\xg\in\Sbb^0$.


On the other hand, there is a natural inclusion
\[ \iota\co\AddCR(F,M\uas)\hookrightarrow\SaddC(F,M\uas). \]
An element $\thh\in\SaddC(F,M\uas)$ belongs to $\AddCR(F,M\uas)$ if and only if
\[ \thh_{\xg}\co F(\xg)\to M(\xg) \]
is an $R$-module homomorphism for any $\xg\in\Sbb^0$.

Together with Remark \ref{RemFBurn}, we have the following diagram.
\[
\xy
(-20,12)*+{\MackSR(F_+,M)}="2";
(20,12)*+{\AddCR(F,M\uas)}="4";
(-40.6,0.6)*+{\ovl{\Th}:}="10";
(-20,0)*+{\MackSR(\Obig^R[F],M)}="12";
(20,0)*+{\SaddC(F,M\uas)}="14";
(37,0.4)*+{:\ovl{\Phi}}="16";
{\ar@{^(->}_{d_F^{\sh}} "2";"12"};
{\ar@{^(->}_{\iota} "4";"14"};
{\ar@{<->}_(0.54){\cong} "12";"14"};
\endxy
\]
To show the Proposition \ref{PropFbAdj}, it remains to show that the corresponding elements
\[ \vp\in\MackSR(\Obig^R[F],M))\ \ \text{and}\ \ \thh\in\SaddC(F,M\uas) \]
satisfy
\[ \vp\in d^{\sh}_F(\MackSR(F_+,M))\ \ \Leftrightarrow\ \ \thh\in\AddCR(F,M\uas). \]
Namely, it suffices to show the following {\rm (1)} and {\rm (2)}.
\begin{enumerate}
\item If $\vp\in\MackSR(F_+,M)$ comes from $\MackSR(\Obig^R[F],M)$, then the map $\ovl{\Th}(\vp)_{\xg}\co F(\xg)\to M(\xg)$ becomes an $R$-homomorphism  for any $\xg\in\Ob(\Csc)$.
\item For any $\thh\in\AddCR(F,M\uas)$, the map $\ovl{\Phi}(\thh)_{\xg}\co \Obig^R[F](\xg)\to M(\xg)$ 
satisfies $\ovl{\Phi}(\thh)_{\xg}(I_F(\xg))=0$.
\end{enumerate}


\medskip

\noindent {\bf Confirmation of {\rm (1)}.}

Suppose $\vp\in\MackSR(F_+,M)$ satisfies
\[ \vp=\psi\ci d_F \]
for some $\psi\in\MackSR(\Obig^R[F], M)$. Let $\xg\in\Sbb^0$ be any 0-cell. Then for any pair of elements $\xi_1,\xi_2\in F(\xg)$ and any $r_1,r_2\in R$, we have
\begin{eqnarray*}
\ovl{\Th}(\vp)_{\xg}(r_1\xi_1+r_2\xi_2)&=&\vp_{\xg}([\xgixg,r_1\xi_1+r_2\xi_2])\\
&=&\psi_{\xg}\ci d^{(F)}_{\xg}([\xgixg,r_1\xi_1+r_2\xi_2])\\
&=&\psi_{\xg}\ci d^{(F)}_{\xg}\big(r_1[\xgixg,\xi_1]+r_2[\xgixg,\xi_2]\big)\\
&=&\vp_{\xg}(r_1[\xgixg,\xi_1]+r_2[\xgixg,\xi_2])\\
&=&r_1\vp_{\xg}([\xgixg,\xi_1])+r_2\vp_{\xg}([\xgixg,\xi_2])\\
&=&r_1\ovl{\Th}(\vp)_{\xg}(\xi_1)+r_2\ovl{\Th}(\vp)_{\xg}(\xi_2).
\end{eqnarray*}


\noindent {\bf Confirmation of {\rm (2)}.}

Let $\thh\in\AddCR(F,M\uas)$ be any element, and let $\xg\in\Sbb^0$ be any 0-cell. Take any generator
\[ \xi=[\akaxg,r_1\xi_1+r_2\xi_2]-\big( r_1[\akaxg,\xi_1]+r_2[\akaxg,\xi_2]\big)  \]
of $I_F(\xg)$, as in $(\ref{Eq_elemI})$.
Then we have
\begin{eqnarray*}
\ovl{\Phi}(\thh)_{\xg}(\xi)&=&\ovl{\Phi}(\thh)_{\xg}([\akaxg,r_1\xi_1+r_2\xi_2])\\
&\ &\ -r_1\ovl{\Phi}(\thh)_{\xg}([\akaxg,\xi_1])-r_2\ovl{\Phi}(\thh)_{\xg}([\akaxg,\xi_2])\\
&=&M\sas(\afr)\thh_{\ak}(r_1\xi_1+r_2\xi_2)-r_1M\sas(\afr)\thh_{\ak}(\xi_1)-r_2M\sas(\afr)\thh_{\ak}(\xi_2)\\
&=&0,
\end{eqnarray*}
since $M\sas(\afr)\ci\thh_{\ak}$ is an $R$-homomorphism.
\end{proof}

\begin{cor}
By composing functors
\[ \AddCR\ov{(-)_+}{\lra}\MackSR\ov{\wt{(-)}}{\lra}\MackdSR\ov{\simeq}{\lra}\BisetFtr^R, \]
we obtain a functor $\AddCR\to\BisetFtr^R$.
\end{cor}
By construction, this sends the trivial functor $0\in\Ob(\AddCR)$ to the ordinary Burnside functor.


\section{Equivalence of categories $\AddC\simeq\ResC$.}

Let $\FG$ denote the category of finite groups and group homomorphisms. Define its \lq stabilization' as follows.

\begin{dfn}\label{DefFinStable}
Define a category $\sFG$ as follows.
\begin{itemize}
\item[-] $\Ob(\sFG)=\Ob(\FG)$.
\item[-] For any $G,H\in\Ob(\sFG)$, the set of morphisms $\sFG(G,H)$ is defined to be the quotient set
\[ \sFG(G,H)=\FG(G,H)/\sim, \]
where two homomorphisms $f,f\ppr\in\FG(G,H)$ are equivalent $(f\sim f\ppr)$ if and only if there exists an element $h\in H$ satisfying $f\ppr=\sigma_h\ci f$, i.e.,
\[ f\ppr(x)=h\cdot f(x)\cdot h\iv\quad(\fa x\in G). \]
We denote the equivalence class of $f$ in $\sFG$ by $\und{f}$.
\end{itemize}
\end{dfn}

\begin{rem}\label{RemFinStable}
Remark that we have a fully faithful functor
\[ \phi\co\sFG\to\Csc, \]
which sends $G\ov{\und{f}}{\lra}H$ in $\sFG$ to $\ptg\ov{\ptf}{\lra}\pth$ in $\Csc$.
\end{rem}

\begin{dfn}\label{DefResFtr}
In analogy with the case of ordinary restriction functor (\cite{Boltje}), we denote the category of contravariant functors from $\sFG$ to $\RMod$ simply by
\[ \ResCR=\FFA. \]
We call an object $P\in\Ob(\ResCR)$ a {\it restriction functor}.
\end{dfn}

\begin{rem}\label{RemEquivAdd}
Remark that $\ResCR$ is equivalent to the full subcategory of the category $\Fun(\FG,\RMod)$, consisting of contravariant functors $P\co\FG\to\RMod$ which satisfy
\[ P(\sig_g)=\id_G \]
for any conjugation homomorphism $\sig_g\co G\to G$ associated to $g\in G$.
\end{rem}


\begin{prop}\label{PropEquivAdd}
There is an equivalence of categories
\[ \AddCR\ov{\cong}{\lra}\ResCR. \]
\end{prop}
\begin{proof}
The proof goes completely in the same manner as in \cite{N_BrRng}. Define functors
\[ \Rbb\co \AddCR\to\ResCR \quad \text{and}\quad \Fbb\co \ResCR\to\AddCR \]
as follows.


\medskip

{\rm (1)} $\Rbb$ is the functor induced by the composition by $\phi$. Namely, for any $F\ov{\eta}{\lra}F\ppr$ in $\AddCR$, we define
\[ \Rbb(F)\ov{\Rbb(\eta)}{\lra}\Rbb(F\ppr) \]
to be
\[ F\ci\phi\ov{\eta\ci\phi}{\lra}F\ppr\ci\phi. \]

{\rm (2)} For any restriction functor $P\in\Ob(\ResCR)$, the functor
\[ \Fbb(P)=\Fbb_P\co\Csc\to\RMod \]
is defined as follows.
\begin{itemize}
\item[{\rm (i)}] For each $\xg\in\Ob(\Csc)$, take a set of representatives $x_1,\ldots,x_s\in X$ of orbits in $X$. Then for each $1\le i\le s$, there is an adjoint equivalence
\[ \zeta_{\xg}^i\co \frac{\pt}{G_{x_i}}\ov{\simeq}{\lra}\frac{Gx_i}{G}, \]
which yields an adjoint equivalence
\[ \zeta_{\xg}=\un{1\le i\le s}{\bigcup}\zeta_{\xg}^i\co \frac{\pt}{G_{x_1}}\am\cdots\am \frac{\pt}{G_{x_s}}\ov{\simeq}{\lra}\frac{Gx_1}{G}\am\cdots\am\frac{Gx_s}{G}\ov{\simeq}{\lra}\xg. \]
Define $\Fbb_P(\xg)$ by
\[ \Fbb_P(\xg)=P(G_{x_1})\oplus\cdots\oplus P(G_{x_s}). \]

\item[{\rm (ii)}] Let $\xg,\yh\in\Ob(\Csc)$ be any pair of objects, with the sets of representatives of orbits $x_1,\ldots,x_s\in X$ and $y_1,\ldots,y_t\in Y$ and morphisms $\zeta_{\xg},\zeta_{\yh}$ chosen in {\rm (i)}. Let $\al\co\xg\to\yh$ be any 1-cell in $\Sbb$. Since a 1-cell preserves orbits, for each $1\le i\le s$, there exists a unique $j_i$ satisfying $\al(Gx_i)\subseteq Hy_{j_i}$. This yields a 1-cell
\[ \al_i=\big( \frac{\pt}{G_{x_i}}\ov{\zeta_{\xg}^i}{\lra}\frac{Gx_i}{G}\ov{\al|_{Gx_i}}{\lra}\frac{Hy_{j_i}}{H}\ov{(\zeta_{\yh}^{j_i})\iv}{\lra}\frac{\pt}{H_{y_{j_i}}} \big). \]
Define $\omega_{ij}$ by
\[ \omega_{ij}=\begin{cases}\, P(\al_i)& j=j_i\\ \,0& \text{otherwise} \end{cases}. \]
Using this, we define
\[ \Fbb_P(\al)\co \Fbb_P(\yh)\to\Fbb_P(\xg) \]
to be the matrix
\[ {[}\omega_{ij}]_{i,j}\co P(H_{y_1})\oplus\cdots\oplus P(H_{y_t})\to P(G_{x_1})\oplus\cdots\oplus P(G_{x_s}). \]
\end{itemize}


It is straightforward to check those $\Rbb$ and $\Fbb$ in fact form functors.
We only show they are mutually quasi-inverse to each other.

For any $P\in\Ob(\ResCR)$ and any $G\in\Ob(\sFG)$, we have
\[ \Rbb(\Fbb(P)(G))=\Fbb_P(\ptg)=P(G), \]
and this gives an isomorphism
\[ \Rbb(\Fbb(P))\ov{\cong}{\lra}P, \]
which is natural in $P$.

Conversely, for any $F\in\Ob(\AddCR)$ and any $\xg\in\Ob(\Csc)$ with the set of representatives of orbits $x_1,\ldots,x_s\in X$, we have an adjoint equivalence $\zeta_{\xg}$ as in {\rm (i)} above,
which yields an isomorphism
\begin{eqnarray*}
\Fbb(\Rbb(F))(\xg)&=&\Rbb(F)(G_{x_1})\oplus\cdots\oplus \Rbb(F)(G_{x_s})\\
&=&F(\frac{\pt}{G_{x_1}})\oplus\cdots \oplus F(\frac{\pt}{G_{x_s}})\ \ 
\underset{F(\zeta_{\xg})}{\ov{\cong}{\lla}}\ \ F(\xg).
\end{eqnarray*}
These form an isomorphism
\[ \Fbb(\Rbb(F))\ov{\cong}{\lra} F, \]
and it can be confirmed to be natural in $F$.
\end{proof}

\begin{dfn}\label{Deftr}
Let $G,H$ be finite groups. For a homomorphism $f\co G\to H$, we denote the bisets
\[ \HHG=\Inddef_f,\quad\text{and}\quad \GHH=\Infres_f \]
by $\tbf(f)$ and $\rbf(f)$, respectively.
\end{dfn}

\begin{rem}\label{Remtr}
If $f,f\ppr\co G\to H$ satisfy $\und{f}=\und{f}\ppr$ in $\sFG$, then there are isomorphisms of bisets
\[ \tbf(f)\cong\tbf(f\ppr)\quad\text{and}\quad\rbf(f)\cong\rbf(f\ppr). \]

Remark that this $\rbf$ defines a contravariant functor
\[ \rbf\co \sFG\to \Bcal, \]
and thus by composition, a functor
\[ \rbf^{\sh}=(-\circ\rbf)\co \BisetFtr^R\simeq\Add(\Bcal,\RMod)\to\ResCR. \]
\end{rem}

\begin{dfn}\label{DefOfmaru}
By composing functors obtained so far, we define a functor $(-)_{\maru}\co\ResCR\to\BisetFtr^R$ to be the composition of
\[ \ResCR\ov{\Fbb}{\lra}\AddCR\ov{(-)_+}{\lra}\MackSR\ov{\wt{(-)}}{\lra}\MackdSR\ov{\simeq}{\lra}\BisetFtr^R. \]
\end{dfn}

\begin{cor}
The functor $(-)_{\maru}\co\ResCR\to\BisetFtr^R$ is a left adjoint to
\[ \rbf^{\sh}\co\BisetFtr^R\to\ResCR. \]
\end{cor}
\begin{proof}
This follows from Corollary \ref{CorMtAdj}, Proposition \ref{PropFbAdj} and \ref{PropEquivAdd}, since there is a commutative diagram of functors as follows.
\[
\xy
(-32,10)*+{\AddCR}="0";
(4,10)*+{\MackSR}="2";
(32,10)*+{\MackdSR}="4";
(-32,-10)*+{\ResCR}="6";
(32,-10)*+{\BisetFtr^R}="8";
{\ar_{M\uas\mapsfrom (M\uas,M\sas)} "2";"0"};
{\ar@{_(->} "4";"2"};
{\ar_{\Rbb}^{\simeq} "0";"6"};
{\ar^{\simeq} "4";"8"};
{\ar^{\rbf^{\sh}} "8";"6"};
{\ar@{}|\circlearrowright "0";"8"};
\endxy
\]
\end{proof}

\section{Direct construction of $(-)_{\maru}\co \ResC\to\BisetFtr$.}


\begin{dfn}\label{DefEplusEachGX}
Let $P\co\sFG\to\RMod$ be any object in $\ResCR$. For each finite group $G$, define $R$-modules $S_P(G)$, $N_P(G)$ and $\Pm(G)$ as follows.
\begin{enumerate}
\item Define $S_P(G)$ by 
\[ S_P(G)=\bigoplus_{K\ov{f}{\to}G}P(K), \]
where the direct sum runs over all group homomorphisms $K\ov{f}{\to}G$. For $K\ov{f}{\to}G$ and $\kp\in P(K)$, the corresponding element in $S_P(G)$ is denoted by $(\kfgk)$ or simply by $(\fk)$
\item $N_P(G)\subseteq S_P(G)$ is the submodule generated by the following subset of $S_P(G)$.
\[ \Set{(\fk)-(f\ppr,P(\pi)(\kp))|\ %
\xy
(-7,6)*+{K\ppr}="0";
(7,6)*+{K}="2";
(-7,-6)*+{G}="4";
(7,-6)*+{G}="6";
{\ar^{\pi} "0";"2"};
{\ar_{f\ppr} "0";"4"};
{\ar^{f} "2";"6"};
{\ar_{\sig_g} "4";"6"};
{\ar@{}|\circlearrowright "0";"6"};
\endxy
\ \text{in}\ \FG,
\begin{array}{l}
\pi\ \text{is surjective},\\
g\in G,\ \kp\in P(K).
\end{array}}
\]
\item $\Pm(G)$ is the quotient of $S_P(G)$ by $N_P(G)$:
\[ \Pm(G)=S_P(G)/N_P(G). \]
We denote the equivalence class of $(\kfgk)$ in $\Pm(G)$ by $[\kfgk]$, or simply by $[\fk]$.

\end{enumerate}
\end{dfn}


\begin{dfn}\label{DefContraction}
Let $H,K$ be finite groups.
\begin{enumerate}
\item A {\it span} to $H$ from $K$ in $\FG$ is defined to be a triplet $(q,\Gcal,p)$ of
\begin{itemize}
\item[-] a finite group $\Gcal$,
\item[-] homomorphisms $p\co\Gcal\to K$ and $q\co \Gcal\to H$.
\end{itemize}
\item Let $(q,\Gcal,p)$ and $(q\ppr,\Gcal\ppr,p\ppr)$ be spans to $H$ from $K$. A {\it contraction}
\[ (h,\pi,k)\co (q,\Gcal,p)\thra (q\ppr,\Gcal\ppr,p\ppr) \]
is a triplet of
\begin{itemize}
\item[-] surjective group homomorphism $\pi\co\Gcal\to\Gcal\ppr$,
\item[-] elements $h\in H$ and $k\in K$,
\end{itemize}
which makes the following diagram commutative.
\begin{equation}\label{*}
\xy
(-12,12)*+{\Gcal}="0";
(17,12)*+{K}="2";
(-6,2)*+{}="3";
(-2,2)*+{\Gcal\ppr}="4";
(17,2)*+{K}="6";
(-12,-14)*+{H}="8";
(-2,-14)*+{H}="10";
{\ar^{p} "0";"2"};
{\ar@{->>}_{\pi} "0";"4"};
{\ar_{p\ppr} "4";"6"};
{\ar^{\sig_k} "2";"6"};
{\ar_{q} "0";"8"};
{\ar^{q\ppr} "4";"10"};
{\ar_{\sig_h} "8";"10"};
{\ar@{}|\circlearrowright "2";"3"};
{\ar@{}|\circlearrowright "4";"8"};
\endxy
\end{equation}
\end{enumerate}
\end{dfn}


\begin{lem}\label{LemContraction}
Let $P\in\Ob(\ResCR)$ be any object, let $H,K$ be finite groups, and let
\[ (h,\pi,k)\co (q,\Gcal,p)\thra(q\ppr,\Gcal\ppr,p\ppr) \] 
be a contraction.
Then for any $\kp\in P(K)$, the equality
\[ [\Gcal\ov{q}{\to}H,P(p)(\kp)]=[\Gcal\ppr\ov{q\ppr}{\to}H,P(p\ppr)(\kp)] \]
holds in $\Pm(H)$.
\end{lem}
\begin{proof}
By the commutativity of $(\ref{*})$, This follows from
\begin{eqnarray*}
[q,P(p)(\kp)]&=&[q,P(p)P(\sig_k)(\kp)]\\
&=&[q,P(\pi)(P(p\ppr)(\kp))]\ =\ [q\ppr,P(p\ppr)(\kp)].
\end{eqnarray*}
\end{proof}


\begin{dfn}\label{DefGandD}
Let $G,H,K$ be finite groups. Let $\HUG$ be an $H$-$G$-biset, and let $f\co K\to G$ be a group homomorphism.
\begin{enumerate}
\item For any $u\in U$, define its {\it stabilizing span} $(q_u,\Gcal_u(K),p_u)$ of $f$ at $u$ by
\[ \Gcal_u(K)=\{ (h,k)\in H\times K\mid hu=uf(k)\}, \]
\[ q_u((h,k))=h,\ \  p_u((h,k))=k\quad (\fa (h,k)\in\Gcal_u(K)). \]
We depict this as follows.
\[
\xy
(-9,6)*+{\Gcal_u(K)}="0";
(9,6)*+{K}="2";
(-9,-6)*+{H}="4";
(9,-6)*+{G}="6";
{\ar^(0.56){p_u} "0";"2"};
{\ar_{q_u} "0";"4"};
{\ar^{f} "2";"6"};
{\ar@{<..}_{U} "4";"6"};
{\ar@{}|\msp "0";"6"};
\endxy
\]
\item Denote the double coset $H\bs U/K=H\bs U/f(K)$ by $D_U(f)$.
\end{enumerate}
\end{dfn}


\begin{rem}\label{RemGandD}
Let $G,H,K,U,f$ be as in Definition \ref{DefGandD}. If $u,u\ppr\in U$ satisfies $HuK=Hu\ppr K$ in $D_U(f)$, then there is a contraction
\[ (q_u,\Gcal_u(K),p_u)\thra(q_{u\ppr},\Gcal_{u\ppr}(K),p_{u\ppr}). \]
\end{rem}
\begin{proof}
By $HuK=Hu\ppr K$, there are $h_0\in H$ and $k_0\in K$ satisfying $u\ppr=h_0uf(k_0)$. Since the conjugation homomorphism
\[ \sig_{(h_0,k_0)}\co\Gcal_u(K)\ov{\cong}{\lra}\Gcal_{u\ppr}(K)\ ;\ (h,k)\mapsto (h_0hh_0\iv,k_0kk_0\iv) \]
gives an isomorphism compatible with $\sig_{h_0}$ and $\sig_{k_0}$, we have a contraction
\[ (h_0,\sig_{(h_0,k_0)},k_0)\co (q_u,\Gcal_u(K),p_u)\thra (q_{u\ppr},\Gcal_{u\ppr}(K),p_{u\ppr}). \]
\end{proof}


\begin{prop}\label{PropGandD}
Let $G,H,K,U,f$ be as in Definition \ref{DefGandD}. If we take a set of representatives $u_1,\ldots,u_s\in U$ for $D_U(f)=H\bs U/K$, and take stabilizing spans of $f$ at $u_i$
\[ (q_{u_i},\Gcal_{u_i}(K),p_{u_i})\quad(1\le i\le s), \]
then there is an isomorphism of $H$-$K$-bisets
\[ U\un{G}{\times}\tbf(f)\cong \coprod_{1\le i\le s}\tbf(q_{u_i})\un{\Gcal_{u_i}(K)}{\times}\rbf(p_{u_i}). \]
\end{prop}
\begin{proof}
Remark that $U\un{G}{\times}\tbf(f)$ is nothing but the biset $\HUK$, where $K$ acts on $U$ through $f$ as
\[ u\cdot k=uf(k)\quad(\fa u\in U,\fa k\in K). \]
Thus with $u_1,\ldots,u_s$, we have a decomposition
\[ U\un{G}{\times}\tbf(f)=\HUK=\coprod_{1\le i\le s}Hu_iK \]
as an $H$-$K$-biset.


Hence it remains to show that for any $u\in U$, there is an isomorphism of $H$-$K$-bisets
\[ HuK\cong\tbf(q_u)\un{\Gcal_u(K)}{\times}\rbf(p_u). \]
By definition, $\tbf(q_u)\un{\Gcal_u(K)}{\times}\rbf(p_u)=H\un{\Gcal_u(K)}{\times}K$ is a quotient of $H\times K$ by an equivalence relation, which can be rephrased as follows. 
\begin{itemize}
\item[$\ $] \hspace{-0.8cm}$(h,k),(h\ppr,k\ppr)\in H\times K$ are equivalent.
\item[$\LR$] there is $(h_0,k_0)\in\Gcal_u(K)$ satisfying $hh_0=h\ppr$ and $k=k_0k\ppr$.
\item[$\LR$] there is $(h_0,k_0)\in H\times K$ satisfying $hh_0=h\ppr,\ k=k_0k\ppr$ and $h_0u=uf(k_0)$.
\item[$\LR$] $(h,k),(h\ppr,k\ppr)$ satisfy $h\iv h\ppr u=uf(kk^{\prime-1})$.
\item[$\LR$] $(h,k),(h\ppr,k\ppr)$ satisfy $huf(k)=h\ppr uf(k\ppr)$.
\end{itemize}
Thus there is an isomorphism of $H$-$K$-bisets
\[ \tbf(q_u)\un{\Gcal_u(K)}{\times}\rbf(p_u)\to HuK\ ;\ [h,k]\mapsto huf(k). \]
\end{proof}


\begin{prop}\label{PropGandDUV}
Let $G,H,K,L$ be finite groups, let $\HUG$ and $\LVH$ be bisets, and let $f\co K\to G$ be a group homomorphism. Then the following holds.
\begin{enumerate}
\item Let $u\in U$ and $v\in V$ be any pair of elements, and put $w=[v,u]\in\VU$. Take stabilizing spans
\begin{itemize}
\item[-] $(q_u,\Gcal_u(K),p_u)$ of $K\ov{f}{\to}G$ at $u$,
\item[-] $(q_v,\Gcal_v(\Gcal_u(K)),p_v)$ of $\Gcal_u(K)\ov{q_u}{\to}H$ at $v$,
\item[-] $(q_w,\Gcal_w(K),p_w)$ of $K\ov{f}{\to}G$ at $w$.
\end{itemize}
Then there exists a contraction
\[ (q_v,\Gcal_v(\Gcal_u(K)),p_u\ci p_v)\thra (q_w,\Gcal_w(K),p_w). \]
This can be depicted as follows.
\[
\xy
(-40,20)*+{\Gcal_w(K)}="-2";
(-10,20)*+{}="1";
(-36,-2)*+{}="-1";
(-26,8)*+{\Gcal_v(\Gcal_u(K))}="0";
(0,8)*+{\Gcal_u(K)}="2";
(24,8)*+{K}="4";
(-26,-8)*+{L}="10";
(0,-8)*+{H}="12";
(24,-8)*+{G}="14";
{\ar_{p_v} "0";"2"};
{\ar_{p_u} "2";"4"};
{\ar^{q_v} "0";"10"};
{\ar^{q_u} "2";"12"};
{\ar^{f} "4";"14"};
{\ar@{<..}_{V} "10";"12"};
{\ar@{<..}_{U} "12";"14"};
{\ar@/^0.80pc/^{p_w} "-2";"4"};
{\ar@{->>}^{} "0";"-2"};
{\ar@/_0.80pc/_{q_w} "-2";"10"};
{\ar@{}|\circlearrowright "0";"1"};
{\ar@{}|\circlearrowright "0";"-1"};
{\ar@{}|\msp "0";"12"};
{\ar@{}|\msp "2";"14"};
\endxy
\]

\item Let $u_1,\ldots,u_s\in U$ be a set of representatives for $D_U(f)$, and let $v_{i1},\ldots,v_{it_i}\in V$ be a set of representative for $D_V(q_{u_i})$, for each $1\le i\le s$. If we put $w_{ij}=[v_{ij},u_i]$, then
\[ \{ w_{ij}\in V\un{H}{\times}U\mid 1\le i\le s,\ 1\le j\le t_i \} \]
gives a set of representatives for $D_{\VU}(f)$.
\end{enumerate}
\end{prop}
\begin{proof}
\noindent {\rm (1)} We have
\begin{eqnarray*}
\Gcal_v(\Gcal_u(K))&=&\{ (\ell,(h,k))\in L\times\Gcal_u(K)\mid \ell v=vh \}\\
&=&\{ (\ell,h,k)\in L\times H\times K\mid \ell v=vh,\ hu=uf(k) \},
\end{eqnarray*}
and
\begin{eqnarray*}
\Gcal_{w}(K)&=&\{ (\ell,k)\in L\times K\mid \ell [v,u]=[v,u]f(k) \}\\
&=&\{ (\ell,k)\in L\times K\mid \ell v=vh,\ hu=uf(k)\ \ \text{for some}\ h\in H \}.
\end{eqnarray*}
If we define
\[ \pi\co\Gcal_v(\Gcal_u(K))\to\Gcal_w(K) \]
by $\pi((\ell,h,k))=(\ell,k)$, this gives a contraction
\[ (e,\pi,e)\co (q_v,\Gcal_v(\Gcal_u(K)),p_u\ci p_v)\thra (q_w,\Gcal_w(K),p_w). \]
\bigskip

\noindent {\rm (2)}
First we show $\VU=\un{i,j}{\bigcup}Lw_{ij}f(K)$.
Take any $[v,u]\in\VU$. Since $U=\un{1\le i\le s}{\coprod}Hu_if(K)$, there exists $1\le i\le s$ and $(h,k)\in H\times K$ satisfying
\[ u=hu_if(k). \]
Then for $vh\in V$, since $V=\un{1\le j\le t_i}{\coprod}Lv_{ij}q_{u_i}(\Gcal_{u_i}(K))$, there exist $1\le j\le t_i$ and $(\ell,h\ppr,k\ppr)\in L\times H\times K$ satisfying
\[ vh=\ell v_{ij}h\ppr,\ \ h\ppr u_i=u_if(k\ppr). \]
Thus we have
\begin{eqnarray*}
[v,u]&=&[v,hu_if(k)]\ =\ [\ell v_{ij}h\ppr,u_if(k)]\\
&=&[\ell v_{ij},u_if(k\ppr)f(k)]\ =\ \ell [v_{ij},u_i]f(k\ppr k).
\end{eqnarray*}
This means $\VU=\un{i,j}{\bigcup}L[v_{ij},u_i]f(K)$.

It remains to show that this union is disjoint. Suppose
\[ w_{i\ppr j\ppr}\in Lw_{ij}f(K) \]
holds for some $i,j$ and $i\ppr,j\ppr$.
Then there exists $\ell\in L,k\in K$ and $h\in H$ satisfying
\[ (v_{i\ppr j\ppr},u_{i\ppr})=(\ell v_{ij}h,h\iv u_if(k)). \]
Thus $u_{i\ppr}=h\iv u_if(k)$ implies $i\ppr=i$, and then $v_{ij\ppr}=\ell v_{ij}h$ implies $j\ppr=j$.
\end{proof}

\begin{dfn}\label{DefEplusEachUandf}
Let $P$ be an object in $\ResCR$, let $\HUG$ be any $H$-$G$-biset. For any element $[\kfgk]\in \Pm(G)$, define $\Pm(U)([\kfgk])\in \Pm(H)$ as follows.
\begin{itemize}
\item Choose a set of representatives $u_1,\ldots,u_s\in U$ for $D_U(f)$, and take stabilizing span $(q_{u_i},\Gcal_{u_i}(K),p_{u_i})$ for each $1\le i\le s$. Then, define as
\begin{equation}\label{Eq_Eplus}
\Pm(U)([\kfgk])=\sum_{1\le i\le s}[\Gcal_{u_i}(K)\ov{q_{u_i}}{\to}H,P(p_{u_i})(\kp)].
\end{equation}
\end{itemize}
\end{dfn}


\begin{rem}\label{RemEplusEachUandf}
Let $P,U$ and $[\kfgk]\in\Pm(G)$ be as above.
\begin{enumerate}
\item For each $(\kfgk)\in S_P(G)$, the right hand side of $(\ref{Eq_Eplus})$ does not depend on the choice of the set of representatives $u_1,\ldots,u_s\in U$.
\item $(\ref{Eq_Eplus})$ gives a well-defined module homomorphism $\Pm(U)\co \Pm(G)\to \Pm(H)$.
\item If $U\cong U\ppr$ holds for $H$-$G$-bisets $U$ and $U\ppr$, then $\Pm(U)([\fk])=\Pm([\fk])$ holds.
\item For arbitrary $H$-$G$-bisets $U$ and $U\ppr$,
\[ \Pm(U\am U\ppr)([\fk])=\Pm(U)([\fk])+\Pm(U\ppr)([\fk]) \]
holds. Namely, we have $\Pm(U\am U\ppr)=\Pm(U)+\Pm(U\ppr)$.
\end{enumerate}
\end{rem}
\begin{proof}
{\rm (1)} Let $u_1,\ldots,u_s\in U$ and $u_1\ppr,\ldots,u_s\ppr\in U$ be two choices of sets of representatives for $D_U(f)$. Renumbering $u_1\ppr,\ldots,u_s\ppr$ if necessary, we may assume that $Hu_iK=Hu_i\ppr K$ holds for each $1\le i\le s$. Then there is a contraction
\[ (h_i,\sig_{(h_i,k_i)},k_i)\co (q_{u_i},\Gcal_{u_i}(K),p_{u_i})\thra (q_{u_i\ppr},\Gcal_{u_i\ppr}(K),p_{u_i\ppr}) \]
for any $1\le i\le s$ by Remark \ref{RemGandD}, and thus we have
\[ [\Gcal_{u_i}(K)\ov{q_{u_i}}{\to}H,P(p_{u_i})(\kp)]=[\Gcal_{u_i\ppr}(K)\ov{q_{u_i\ppr}}{\to}H,P(p_{u_i\ppr})(\kp)] \]
by Lemma \ref{LemContraction}.


By a similar argument, {\rm (2), (3)} will follow from the following claim, which can be confirmed easily.
\begin{claim}\label{ClaimEplusEachUandf}
Let $\nu\co U\ov{\cong}{\lra}U\ppr$ be an isomorphism of $H$-$G$-bisets,  and let
\[
\xy
(-7,6)*+{K\ppr}="0";
(7,6)*+{K}="2";
(-7,-6)*+{G}="4";
(7,-6)*+{G}="6";
{\ar^{\pi} "0";"2"};
{\ar_{f\ppr} "0";"4"};
{\ar^{f} "2";"6"};
{\ar_{\sig_g} "4";"6"};
{\ar@{}|\circlearrowright "0";"6"};
\endxy\quad(g\in G)
\]
be any commutative diagram in $\FG$, with $\pi$ surjective. Then the following holds.
\begin{itemize}
\item[{\rm (i)}] If $u_1,\ldots,u_s\in U$ is a set of representatives for $D_U(f)$, then $\nu(u_1)g,\ldots,\nu(u_s)g\in U\ppr$ gives a set of representatives for $D_{U\ppr}(f\ppr)$.
\item[{\rm (ii)}] For any $u\in U$, take stabilizing spans
\begin{itemize}
\item[-] $(q_u,\Gcal_u(K),p_u)$ at $u\in U$ of $K\ov{f}{\to}G$,
\item[-] $(q_{\nu(u)g},\Gcal_{\nu(u)g}(K\ppr),p_{\nu(u)g})$ at $\nu(u)g\in U\ppr$ of $K\ppr\ov{f\ppr}{\to}G$.
\end{itemize}
Then, the group homomorphism
\[ \pi^{\dag}\co\Gcal_{\nu(u)g}(K\ppr)\to\Gcal_u(K)\ ;\ (h,k\ppr)\mapsto (h,\pi(k\ppr)) \]
gives a contraction
\[ (e,\pi^{\dag},e)\co (q_{\nu(u)g},\Gcal_{\nu(u)g}(K\ppr),\pi\ci p_{\nu(u)g})\thra (q_u,\Gcal_u(K),p_u). \]
\end{itemize}
\end{claim}

\bigskip

In fact, by Lemma \ref{LemContraction} and Claim \ref{ClaimEplusEachUandf}, it follows
\begin{eqnarray*}
\Pm(U)([\kfgk])&=&\sum_{1\le i\le s}[\Gcal_{u_i}(K)\ov{q_{u_i}}{\to}H,P(p_{u_i})(\kp)]\\
&=&\sum_{1\le i\le s}[\Gcal_{\nu(u_i)g}(K)\ov{q_{\nu(u_i)g}}{\to}H,P(p_{\nu(u_i)g})P(\pi)(\kp)]\\
&=&\Pm(U\ppr)([K\ppr\ov{f\ppr}{\to}G,P(\pi)(\kp)]).
\end{eqnarray*}

{\rm (4)} This is trivial, since for any sets of representatives $u_1,\ldots,u_s\in U$ for $D_U(f)$ and $u_1\ppr,\ldots,u_t\ppr\in U\ppr$ for $D_{U\ppr}(f)$, their union
\[ u_1,\ldots,u_s,u_1\ppr,\ldots,u_t\ppr \in U\am U\ppr\]
gives a set of representatives for $D_{U\am U\ppr}(f)$.
\end{proof}


\begin{rem}\label{RemSp}
Let $P\in\Ob(\ResCR)$ be any restriction functor, let $f\co G\to H$ be any group homomorphism, and let $\kp\in E(G)$ and $\eta\in E(H)$ be any element. The following holds.
\begin{enumerate}
\item $\Pm(\tbf(f))([G\ov{\id_G}{\to}G,\kp])=[G\ov{f}{\to}H,\kp]$.
\item $\Pm(\rbf(f))([H\ov{\id_H}{\to}H,\eta])=[G\ov{\id}{\to}G,P(f)(\eta)]$.
\end{enumerate}
Especially, we have $\Pm({}_G\Id_G)([G\ov{\id_G}{\to}G,\kp])=[G\ov{\id}{\to}G,\kp]$.
\end{rem}

\begin{lem}\label{LemEplusEachUVf}
Let $P$ be an object in $\ResCR$. Let $G,H,K,L$ be finite groups, let $\HUG$ and $\LVH$ be bisets, and let $K\ov{f}{\to}G$ be any homomorphism. Then for any $\kp\in P(K)$, we have
\[ \Pm(\VU)([\kfgk])=\Pm(V)\Pm(U)([\kfgk]). \]
\end{lem}
\begin{proof}
Let $u_1,\ldots,u_s\in U$ be a set of representatives for $D_U(f)$, and let $v_{i1},\ldots,v_{it_i}\in V$ be a set of representatives for $D_V(q_{u_i})$. If we put $w_{ij}=[v_{ij},u_i]$, then by Lemma \ref{LemContraction} and Proposition \ref{PropGandDUV}, we have
\begin{eqnarray*}
\Pm(\VU)([\kfgk])&=&\sum_{1\le i\le s}\sum_{1\le j\le t_i}[\Gcal_{w_{ij}}(K)\ov{q_{w_{ij}}}{\lra}L,P(p_{w_{ij}})(\kp)]\\
&=&\sum_{1\le i\le s}\sum_{1\le j\le t_i}[\Gcal_{v_{ij}}(\Gcal_{u_i}(K))\ov{q_{v_{ij}}}{\to}L,P(p_{v_{ij}})P(p_{u_i})(\kp)]\\
&=&\Pm(V)\big(\sum_{1\le i\le s}[\Gcal_{u_i}(K)\ov{q_{u_i}}{\to}H,P(p_{u_i})(\kp)]\big)\\
&=&\Pm(V)\Pm(U)([\kfgk]).
\end{eqnarray*}
\end{proof}


By Remark \ref{RemEplusEachUandf}, we can define as follows.
\begin{dfn}\label{DefEplusEachMorph}
Let $P$ be any object in $\ResCR$ and let $G,H$ be any pair of finite groups. Extending $(\ref{Eq_Eplus})$ by linearity, we obtain a well-defined module homomorphism
\[ \Pm\co\Bcal(G,H)\to\RMod(\Pm(G),\Pm(H)). \]
\end{dfn}

\begin{prop}\label{PropEplusFunct}
Let $P$ be any object in $\ResCR$. With the correspondences defined in Definition \ref{DefEplusEachGX} and Definition \ref{DefEplusEachMorph},
\[ \Pm\co\Bcal\to\RMod \]
forms an additive functor. Namely, $\Pm$ becomes a biset functor.
\end{prop}
\begin{proof}
This immediately follows from Remark \ref{RemEplusEachUandf}, \ref{RemSp} and Lemma \ref{LemEplusEachUVf}.
\end{proof}

\begin{lem}\label{LemForEmAdjEachObj}
Let $P\in\Ob(\ResCR)$ be any object. For each finite group $G$, define a homomorphism $\delta_G^{(P)}\co P(G)\to \Pm(G)$ by
\[ \delta_G^{(P)}(\kp)=[G\ov{\id_G}{\to}G,\kp]. \]
Then $\delta^{(P)}=\{ \delta^{(P)}_G \}_{G\in\Ob(\FG)}$ gives a morphism
\[ \delta^{(P)}\co P\to\rbf^{\sh}(\Pm) \]
in $\ResCR$.
\end{lem}
\begin{proof}
For any group homomorphism $f\co G\to H$ and any $\kp\in P(H)$, we have
\begin{eqnarray*}
\Pm(\rbf(f))\delta^{(P)}_G(\kp)&=&\Pm(\rbf(f))([H\ov{\id_H}{\to}H,\kp])\\
&=&([G\ov{\id_G}{\to}G,P(f)(\kp)])\ =\ \delta^{(P)}_HP(f)(\kp)
\end{eqnarray*}
by Remark \ref{RemSp}.

\end{proof}

\begin{prop}\label{PropEmAdjEachObj}
For any $P\in\Ob(\ResCR)$ and $B\in\Ob(\BisetFtr^R)$, there is a natural bijection
\[ \BisetFtr^R(\Pm,B)\ov{\cong}{\lra}\ResCR(P,\rbf^{\sh}B). \]
\end{prop}
\begin{proof}
We construct maps
\[ \Xi\co\BisetFtr^R(\Pm,B)\to\ResCR(P,\rbf^{\sh}B) \]
and
\[ \Lam\co\ResCR(P,\rbf^{\sh}B)\to\BisetFtr^R(\Pm,B), \]
and show that they are inverse to each other.

\bigskip

\noindent {\bf {\rm (1)} Construction of $\Xi$.}

For any morphism $\lam\in\BisetFtr^R(\Pm,B)$, define $\Xi(\lam)\co P\to \rbf^{\sh}B$ by
\[ \Xi(\lam)=\rbf^{\sh}(\lam)\ci\delta. \]

\bigskip

\noindent {\bf {\rm (2)} Construction of $\Lam$.}

Let $\xi\in\ResCR(P,\rbf^{\sh}B)$ be any morphism. For any finite group $G$, define a homomorphism $\Lam(\xi)_G\co \Pm(G)\to B(G)$ by
\begin{eqnarray*}
\Lam(\xi)_G([\kfgk])=B(\tbf(f))\xi_K(\kp).
\end{eqnarray*}
For any $H$-$G$-biset $U$, if we take a set of representatives $u_1,\ldots,u_s\in U$ for $D_U(f)$, then we have
\begin{eqnarray*}
\Lam(\xi)_H\ci \Pm(U)([\kfgk])%
&=&\Lam(\xi)_H\big(\sum_{1\le i\le s} [\Gcal_{u_i}(K)\ov{q_{u_i}}{\to}H, P(p_{u_i})(\kp)] \big)\\
&=&\sum_{1\le i\le s} B(\tbf(q_{u_i}))\xi_{\Gcal_{u_i}(K)}P(p_{u_i})(\kp) \\
&=&\sum_{1\le i\le s} B(\tbf(q_{u_i}))B(\rbf(p_{u_i}))\xi_K(\kp) \\
&=&B\big(\coprod_{1\le i\le s}\tbf(q_{u_i})\un{\Gcal_{u_i}(K)}{\times}\rbf(p_{u_i})\big)\xi_K(\kp)\\
&=&B(U\un{G}{\times}\tbf(f))\xi_K(\kp)\\
&=&B(U)B(\tbf(f))\xi_K(\kp)\\
&=&B(U)\ci\Lam(\xi)_G([\kfgk])
\end{eqnarray*}
by Proposition \ref{PropGandD}. This implies the commutativity of
\[
\xy
(-11,6)*+{\Pm(G)}="0";
(9,6)*+{B(G)}="2";
(-11,-6)*+{\Pm(H)}="4";
(9,-6)*+{B(H)}="6";
{\ar^(0.52){\Lam(\xi)_G} "0";"2"};
{\ar_{\Pm(U)} "0";"4"};
{\ar^{B(U)} "2";"6"};
{\ar_(0.52){\Lam(\xi)_H} "4";"6"};
{\ar@{}|\circlearrowright "0";"6"};
\endxy
\]
which shows that $\Lam(\xi)\co \Pm\to B$ is a morphism of biset functors.

\bigskip

\noindent {\bf {\rm (3)} Confirmation of $\Xi\ci\Lam=\id$.}

For any $\xi\in\ResCR(P,\rbf^{\sh}B)$ and $G$, we have
\[ \Xi(\Lam(\xi))_G(\kp)=\Lam(\xi)_G([G\ov{\id_G}{\to}G,\kp])=B(\tbf(\id_G))\xi_G(\kp)=\xi_G(\kp)\quad(\fa\kp\in P(K)). \]
This means $\Xi(\Lam(\xi))=\xi$.

\bigskip

\noindent {\bf {\rm (4)} Confirmation of $\Lam\ci\Xi=\id$.}

For any $\lam$ and $G$, and for any $[\kfgk]\in \Pm(G)$, we have
\begin{eqnarray*}
\Lam(\Xi(\lam))_G([\kfgk])&=&B(\tbf(f))\ci\Xi(\lam)_K(\kp)\\
&=&B(\tbf(f))\ci\lam_K([K\ov{\id_K}{\to}K,\kp])\\
&=&\lam_G\ci \Pm(\tbf(f))([K\ov{\id_K}{\to}K,\kp])\ =\ \lam_G([\kfgk])
\end{eqnarray*}
by Remark \ref{RemSp}.
\end{proof}

\section*{Acknowledgement}
This article has been written when the author was staying at LAMFA, l'Universit\'{e} de Picardie-Jules Verne, by the support of JSPS Postdoctoral Fellowships for Research Abroad. He wishes to thank the hospitality of Professor Serge Bouc, Professor Radu Stancu and the members of LAMFA.

\end{document}